\documentclass{amsart}
\usepackage{amsfonts}

\newtheorem{theorem}{Theorem}
\theoremstyle{plain}

\newtheorem{corollary}{Corollary}

\newtheorem{lemma}{Lemma}

\newtheorem{proposition}{Proposition}
\newtheorem{remark}{Remark}

\numberwithin{equation}{section}
\input{tcilatex}

\begin{document}
\title[Pochhammer identities]{A few finite and infinite identities involving
Pochhammer and $q$-Pochhammer symbols obtained via analytical methods}
\author{Pawe\l\ J. Szab\l owski}
\address{Department of Department of Mathematics \\
and Information Sciences, \\
Warsaw University of Technology, ul Koszykowa 75, 00-662 Warsaw, Poland}
\email{pawel.szablowski@gmail.com}
\date{November , 2024}

\begin{abstract}
We present several identities with a form of polynomials or rational
functions that involve Pochhammer and $q$-Pochhammer symbols and $q$%
-binomials (i.e. Gauss polynomials). All these identities were obtained by
some analytical methods based on infinite expansions of the ratio of
densities in a Fourier series of polynomials orthogonal with respect to the
density in the denominator. We want a unified approach to justify many known
and unknown identities. The purpose of studying these identities is to
simplify calculations occurring while dealing with Pochhammer and $q$%
-Pochhammer symbols. Additional possible applications of the results
presented in the paper are applications within the Combinatorics and the
transformation formulae of hypergeometric and basic hypergeometric functions.
\end{abstract}

\thanks{The author is grateful to an unknown referee for his (her) general
and also quite detailed remarks regarding the improvement of the paper.}
\maketitle

\section{Introduction}

We will follow the ideas presented in \cite{Szablowski2010(1)}, \cite{Szab13}%
, \cite{SzabChol}. Assume that we have two positive, probability measures
say $\mu $ and $\nu $ (most often absolutely continuous with respect to the
Lebesgue measure, i.e., both having densities) and absolutely continuous
with respect to one another. Assume also that we know the two sets of
polynomials $\left\{ \alpha _{n}\right\} $ and $\left\{ \beta _{n}\right\} $
orthogonal with respect to, respectively, $\mu $ and $\nu .$ Within the
paper $\symbol{94}$ over the symbol of orthogonal polynomial $p_{n}$, that
is a member of some family $\left\{ p_{n}\right\} $ would mean the number 
\begin{equation*}
\hat{p}_{n}\allowbreak =\allowbreak \int p_{n}^{2}(x)\mu (dx),
\end{equation*}%
where $\mu $ denotes the probability measure that makes the family $\left\{
p_{n}\right\} $ orthogonal.

Let us denote $\frac{d\mu }{d\nu }$ Radon-Nikodym derivative of $d\mu $ with
respect to $d\nu $ and suppose, that%
\begin{equation*}
\int (\frac{d\mu }{d\nu })^{2}d\nu
\end{equation*}%
is finite.

We need to formulate the general theorem, which will be utilized multiple
times below.

\begin{theorem}
\label{podst}Assume that we have two positive real probability measures $\mu 
$ and $\nu $ that are absolutely continuous with respect to one another and
such that 
\begin{equation}
\int (\frac{d\mu }{d\nu })^{2}d\nu <\infty .  \label{sq_i}
\end{equation}%
Let us also assume that one knows the so-called connection coefficients
between the families of polynomials $\left\{ \alpha _{n}\right\} $ and $%
\left\{ \beta _{n}\right\} $, which are orthogonal with respect to
respectively $\mu $ and $\nu $. That is, we assume that we know the set of
coefficients $\left\{ c_{k,n}\right\} _{n\geq 1,0\leq k\leq n}$ defined by
the relationship%
\begin{equation}
\beta _{n}\allowbreak (x)=\allowbreak \sum_{k=0}^{n}c_{n,k}\alpha _{k}(x).
\label{con1}
\end{equation}%
Then, the following expansion 
\begin{equation}
\frac{d\mu }{d\nu }(x)=\sum_{n\geq 0}a_{n}\beta _{n}(x),  \label{basicEx}
\end{equation}%
where 
\begin{equation*}
a_{n}\allowbreak =\allowbreak c_{n,0}/\hat{\beta}_{n},~~\hat{\beta}%
_{n}\allowbreak =\allowbreak \int \beta _{n}^{2}d\nu ,
\end{equation*}
is convergent in $L^{2}(\mathbb{R},d\nu )$, that is in mean-squares (m-s)
sense (mod $\nu ).$
\end{theorem}

\begin{proof}
First of all, the existence and the mean-squares convergence of the
expansion (\ref{basicEx}) follows the general theory of orthogonal series
(see e.g. \cite{Alexits61}) and, in particular, the formula (\ref{sq_i}).
The justification of the formula for the coefficients $\left\{ a_{n}\right\} 
$ is very simple. Namely, if we multiply both sides of (\ref{basicEx}) by $%
\beta _{n}$ and integrate both sides with respect to $d\nu $. Then we get on
the left-hand side 
\begin{equation*}
\int \beta _{n}(x)\frac{d\mu }{d\nu }(x)d\nu (x)\allowbreak =\allowbreak
\int \beta _{n}(x)d\mu (x)\allowbreak =\allowbreak c_{0,n},
\end{equation*}%
while on the right-hand side we get 
\begin{equation*}
a_{n}\int \beta ^{2}(x)d\nu (x)\allowbreak =\allowbreak a_{n}\hat{\beta}_{n}.
\end{equation*}
\end{proof}

Let us remark that we will use the terms connection coefficient and CC in
exchange.

We know also, that the mean-squares convergence implies that $\sum_{n\geq
0}\left\vert a_{n}\right\vert ^{2}<\infty $. If additionally, we know that 
\begin{equation*}
\sum_{n\geq 0}\left\vert a_{n}\right\vert ^{2}\log ^{2}(n+1)<\infty ,
\end{equation*}%
then, by the Rademacher--Meshov theorem, we deduce that the series in
question converges not only in $L^{2}$, but also almost everywhere with
respect to $d\nu $. For more details see e.g. \cite{Alexits61} or any other
book on analysis large enough, to contain a section on orthogonal series. In
the sequel, all considered densities will be supported only on the bounded
segment $[-r,r]$ with finite $r$. For the proper values of some additional
parameters, all these densities are bounded and hence their ratios will be
square integrable. Thus, we will get the condition $\sum_{n\geq 0}\left\vert
a_{n}\right\vert ^{2}<\infty $ satisfied for free. Moreover, in all cases we
will have $\left\vert a_{n}\right\vert ^{2}\leq \rho ^{n}$ for some $\rho <1$%
. Hence the condition $\sum_{n\geq 0}\left\vert a_{n}\right\vert ^{2}\log
^{2}(n+1)<\infty $ is often also naturally satisfied.

\begin{remark}
Obtaining the set of connection coefficients is a very tedious and difficult
task. They were found for only a few pairs of families of orthogonal
polynomials. There exist algorithms how to get them by either using the
coefficients of the three-term recurrences of the given families of
polynomials (see, e.g., \cite{Maroni08}, Theorem 3.1, or recently \cite%
{LWS23}, Corollary 3.2 ) or deducing them from the moment sequences that
generate the given families of orthogonal polynomials (see, e.g., \cite%
{SzabChol}, section 3).
\end{remark}

Now, having expansions (\ref{con1}) and the relationship that is converse to
(\ref{con1}), i.e., the one: 
\begin{equation}
\alpha _{n}(x)=\allowbreak \sum_{k=0}^{n}\bar{c}_{n,k}\beta _{k}(x),
\label{con2}
\end{equation}%
we know (following the properties of orthogonal polynomials) that for all $%
n\geq 1$, we get%
\begin{eqnarray}
0 &=&\int_{\limfunc{supp}(\mu )}\alpha _{n}(x)d\mu \left( x\right) =\int_{%
\limfunc{supp}(\mu )}\left( \sum_{j=0}^{n}\bar{c}_{j,n}\beta _{j}(x)\right)
\left( \sum_{k\geq 0}a_{k}\beta _{k}(x)\right) d\nu (x)  \notag \\
&=&\sum_{j=0}^{n}\bar{c}_{n,j}a_{j}\hat{\beta}_{j}=\sum_{j=0}^{n}\bar{c}%
_{n,j}c_{j,0}.  \label{02}
\end{eqnarray}%
If we have proven similar expansion for the ratio $\frac{d\nu }{d\mu }(x)$,
provided of course that%
\begin{equation*}
\int \left( \frac{d\nu }{d\mu }\right) ^{2}d\mu <\infty ,
\end{equation*}%
then we have for all $n\geq 1$%
\begin{equation*}
0=\sum_{j=0}^{n}c_{j,n}b_{j}\hat{\alpha}_{j},
\end{equation*}%
where the numbers $\left\{ b_{j}\right\} $ are defined by the expansion:%
\begin{equation*}
\frac{d\nu }{d\mu }(x)=\sum_{n\geq 0}b_{n}\alpha _{n}(x).
\end{equation*}%
Then, by analogy we also get the following identity true for all $n\geq 1$:%
\begin{equation}
0=\sum_{j=0}^{n}c_{n,j}\bar{c}_{j,0}.  \label{01}
\end{equation}

In the sequel we will be using the so-called infinite lower triangular
matrices $A$ understood as the sequence of lower triangular matrices $%
\left\{ A_{n}\right\} _{n\geq 0}$ such that $A_{0}$ is a number, $A_{n}$ is $%
(n+1)\times (n+1)$, an upper left sub-matrix of the matrix $A_{n+1}$. The
inverse of $A$ is understood as the sequence of inverses of matrices $A_{n}$
i.e. $A^{-1}\allowbreak =\allowbreak \left\{ A_{n}^{-1}\right\} $. For
example, we have $C\allowbreak =\allowbreak \left[ c_{n,j}\right] $ means
the lower triangular matrix composed of elements $c_{n,j}$ on the $n-$th, $%
j- $th position, with $n\allowbreak =\allowbreak 0,\ldots $, $j\allowbreak
=\allowbreak 0,\ldots ,n$.

\begin{remark}
Notice, that the coefficients $\left\{ c_{j,n}\right\} $ and $\left\{ \bar{c}%
_{j,n}\right\} $ can be set together into two infinite lower triangular
matrices that are inverse of one another. So it is natural that we have 
\begin{equation*}
0=\sum_{j=0}^{n}c_{j,n}\hat{c}_{0,j}=\sum_{j=0}^{n}\bar{c}_{j,n}c_{0,j}.
\end{equation*}
The reasoning from above gives an analytical sense to these identities. By
the way, the fact that the two lower triangular matrices built of
coefficients $\left\{ c_{n,j}\right\} $ and $\left\{ \bar{c}_{n,j}\right\} $
are inverses of one another, implies that the identities (\ref{01}) and (\ref%
{02}) can be extended. Namely, for all $n>j\geq 0$ we have%
\begin{eqnarray*}
\sum_{k=j}^{n}c_{n,k}\bar{c}_{k,j} &=&0, \\
\sum_{k=j}^{n}\bar{c}_{n,k}c_{k,j} &=&0.
\end{eqnarray*}%
Hence, we can extend (\ref{01}) and (\ref{02}) a bit.
\end{remark}

\begin{remark}
As can be noticed, the crucial role in obtaining the identities mentioned
above, is played by the connection coefficients. The Askey-Wilson family of
polynomials is a large family of polynomials for which these connection
coefficients are known or can be relatively easily obtained. It has been
described in detail in \cite{AW85},\cite{KLS},\cite{IA}. So we will not
define it, just referring the reader to these positions of literature. This
is due to the excellent work of Askey and Wilson who in their paper \cite%
{AW85} provided such a set of connection coefficients for every two members
of the AW family having different $3$ out of $4$ parameters (not counting
the one more parameter called base and usually denoted by $q$).
\end{remark}

\begin{remark}
Another family of polynomials for which the connection coefficients can be
easily obtained are the Jacobi polynomials. It is also known, that the
so-called Beta distribution makes these polynomials orthogonal. This is due
to the formulae provided, e.g., in \cite{IA} section 4.2. In the sequel, we
will prove these formulae for the connection coefficient once more in a
different way for the completeness of the paper.
\end{remark}

Let us remark that the AW family of polynomials will be the source of the
identities involving the $q-$Pochhammer symbol, while the Jacobi family will
be the source of identities involving simply the Pochhammer symbol.

For the sake of completeness of the paper, we will briefly introduce the two
families of polynomials. Let us also remark that the family of Chebyshev
polynomials is a subset of the two considered above families of polynomials.

The Jacobi polynomials seem to be simpler hence they will be analyzed first.
The AW family of polynomials is richer and more complicated and thus will be
considered next. There are $5$ families of orthogonal polynomials from the
AW scheme with increasing numbers of parameters ranging from $0$ to $4$
(without counting the base $q$) and two additional, but related families or
polynomials i.e. Chebyshev polynomials of the second kind and the so-called $%
q-$ultraspherical or Rogers polynomials. Hence, we have theoretically $%
\binom{7}{2}\allowbreak =\allowbreak 21$ pairs $\left\{ \alpha _{n},\beta
_{n}\right\} $ of families of polynomials and consequently at most $21$
identities. But we will not consider all these cases, since many of them
lead to some trivial identities. As the considerations above, state, the
crucial for obtaining these identities are finite expansions of the form (%
\ref{con1}). All mentioned-above families of polynomials from the AW scheme
are defined and their basic properties are described in \cite{AW85}, \cite%
{IA}, \cite{KLS}, \cite{Szab-bAW}, \cite{Szab22}.

The paper is organized as follows. In Section \ref{Jaco}, we introduce Beta
distribution, Jacobi polynomials and the Pochhammer symbol. This section
presents in a concise form connection coefficients between different
families of Jacobi polynomials (precisely $9$ sets) and also $8$ identities
involving Pochhammer polynomials of two variables.

In the next Section \ref{q-series}, we introduce basic notions used in the
so-called $q-$series theory including the $q-$Pochhammer symbol, assorted,
the so-called Askey-Wilson polynomials and we recall connection coefficients
between some families of AW polynomials. This section presents several
useful finite and infinite identities involving the $q-$Pochhammer symbol.
Longer, detailed proofs are moved to Section \ref{Dow}.

\section{Jacobi polynomials and the Pochhammer symbol.\label{Jaco}}

Let us recall the definition of Beta distribution. On one hand, we have the
distribution with the density:%
\begin{equation*}
f(x|a,b)=%
\begin{cases}
0\text{,} & \text{if }x\notin \lbrack 0,1]\text{;} \\ 
x^{a-1}(1-x)^{b-1}/B(a,b)\text{,} & \text{if }0\leq x\leq 1\text{,}%
\end{cases}%
\end{equation*}%
where $B(a,b)$ means the Euler's beta function, which is defined for all
complex $a,b$ such that $\func{Re}(a),\func{Re}(b)>0$.

On the other hand the following function is also called density of beta
distributions. It has the following density:%
\begin{equation*}
h(x|a,b)=%
\begin{cases}
(x+1)^{a-1}(1-x)^{b-1}/(B(a,b)2^{a+b-1})\text{,} & \text{if }\left\vert
x\right\vert \leq 1\text{;} \\ 
0\text{,} & \text{if otherwise.}%
\end{cases}%
\end{equation*}%
It is common knowledge (see, e.g., \cite{Andrews1999}) that the polynomials
that are orthogonal with respect to $h$ are the so-called Jacobi polynomials
defined by the formula:%
\begin{equation}
J_{n}(x|a,b)=\frac{1}{n!}\sum_{m=0}^{n}\binom{n}{m}\left( a+b+n-1\right)
^{\left( m\right) }(b+m)^{(n-m)}(x-1)^{m}/2^{m}.  \label{J1}
\end{equation}%
Following simple change of variables under the integral, we deduce that the
following family of polynomials:%
\begin{equation}
K_{n}(x|a,b)=\frac{1}{n!}\sum_{m=0}^{n}\binom{n}{m}\left( a+b+n-1\right)
^{\left( m\right) }(b+m)^{(n-m)}(x-1)^{m},  \label{Ja}
\end{equation}%
is orthogonal with respect to the distribution with the density $f$. Above,
we used, the so-called rising factorial or Pochhammer symbol (polynomial)
which is defined by 
\begin{equation*}
(x)^{(n)}=x(x+1)\ldots (x+n-1),
\end{equation*}%
for all complex $x$. Notice that we have for all $x\neq 0$ we have%
\begin{equation*}
(x)^{(n)}=\frac{\Gamma (x+n)}{\Gamma (x)}\text{,}
\end{equation*}%
where $\Gamma (x)$ denotes the Euler's gamma function. One can consider also
the so-called falling factorials, denoted by $\left( a\right) _{\left(
n\right) }$ and defined by 
\begin{equation*}
\left( a\right) _{\left( n\right) }\allowbreak =\allowbreak
\prod_{j=0}^{n-1}\left( a-j\right) ,
\end{equation*}%
with $\left( a\right) _{\left( 0\right) }\allowbreak =\allowbreak 1$. Let us
notice that we have 
\begin{equation*}
\left( a\right) ^{\left( n\right) }\allowbreak =\allowbreak \left( -1\right)
^{n}\left( -a\right) _{\left( n\right) }
\end{equation*}%
and of course 
\begin{equation*}
\left( a\right) _{\left( n\right) }\allowbreak =\allowbreak \left( -1\right)
^{n}\left( -a\right) ^{\left( n\right) }.
\end{equation*}

One has to remark that in many popular books on special functions or
orthogonal polynomials, like \cite{Andrews1999},\cite{IA},\cite{KLS} one
uses the notation $\left( a\right) _{n}$ to denote rising factorial. Our
notation is more intuitive. Besides, notice that within this paper indices
of Pochhammer symbols are always in the brackets. Further, $\left( a\right)
_{n}$ will denote something else in the Section \ref{q-series} of the paper.
Namely, it will denote the so-called $q-$Pochhammer symbol, when the the
so-called base is known. For the definition and details, see this section
below.

What is more, the falling factorials are closely connected with the
so-called Stirling numbers of the first and second kind. Namely, the
following expansions are true:%
\begin{eqnarray}
\left( x\right) _{\left( n\right) } &=&\sum_{j=0}^{n}\left( -1\right) ^{n-j}%
\QATOPD[ ] {n}{j}x^{j},  \label{S1} \\
\left( x\right) ^{\left( n\right) } &=&\sum_{j-0}^{n}\QATOPD[ ] {n}{j}x^{j},
\label{S!!} \\
x^{n} &=&\sum_{j=0}^{n}\QATOPD\{ \} {n}{j}\left( x\right) _{\left( j\right)
},  \label{S2}
\end{eqnarray}%
where $\left( -1\right) ^{n-j}\QATOPD[ ] {n}{j}$, $\QATOPD[ ] {n}{j}$, $%
\QATOPD\{ \} {n}{j}$ are called respectively Stirling numbers of the first
kind, unsigned Stirling numbers of the first kind and Stirling numbers of
the second kind. These numbers are very important in combinatorics. They
count, e.g., the number of permutations with disjoined $j$ cycles as the
Stirling numbers of the first kind do, or are closely related to another
families of numbers like Bell or Bernoulli like the Stirling numbers of the
second kind. Symbols $\QATOPD[ ] {n}{j}$ appear only here, and shouldn't be
confused with the symbol $\QATOPD[ ] {n}{j}_{q}$ which means something
different and will be defined and used extensively in the next section.

Let us also remark that for $x$ of the form $x\allowbreak =\allowbreak i/2$,
where $i$ is some integer we have 
\begin{equation*}
\left( i\right) ^{\left( n\right) }=\frac{\left( i+n-1\right) !}{(i-1)!}%
,~(x+1)^{\left( n\right) }=\frac{x+n}{x}\left( x\right) ^{\left( n\right) }%
\text{ and }\left( \frac{1}{2}\right) ^{\left( n\right) }=\frac{(2n-1)!!}{%
2^{n}}.
\end{equation*}

Following formula (4.1.5) of \cite{IA} we deduce that polynomials $\left\{
K_{n}\right\} $ and $\left\{ J_{n}\right\} $ are not monic. The coefficient
by $x^{n}$ in $J_{n}$ is equal to%
\begin{equation}
\frac{\left( a+b+n-1\right) ^{\left( n\right) }}{n!2^{n}}.  \label{lead}
\end{equation}

It is also known that 
\begin{equation*}
\int_{-1}^{1}J_{n}^{2}(x|a,b)h\left( x|a,b\right) dx=\frac{\left( a\right)
^{\left( n\right) }\left( b\right) ^{\left( n\right) }}{n!\left(
a+b+2n-1\right) \left( a+b\right) ^{\left( n-1\right) }}.
\end{equation*}

Let us remark that the particular cases of Jacobi polynomials are the
following.

1. The Chebyshev polynomials of the first kind for $a\allowbreak
=\allowbreak b\allowbreak =\allowbreak 1/2$ are orthogonal with respect to
the so-called arcsine distribution that has the following density%
\begin{equation*}
h\left( x|1/2,1/2\right) =%
\begin{cases}
\frac{1}{\pi \sqrt{1-x^{2}}}\text{,} & \text{if }\left\vert x\right\vert <1%
\text{;} \\ 
0\text{,} & \text{otherwise .}%
\end{cases}%
\end{equation*}%
The traditional denotation for the Chebyshev polynomials of the first kind
is $T_{n}(x)$. It has the leading coefficient equal to $2^{n-1}$. Hence,
following (\ref{lead}), we see that for $n\geq 1$ we have 
\begin{equation*}
T_{n}(x)=J_{n}(x|1/2,1,2)\frac{2^{2n-1}n!n!}{(2n)!}.
\end{equation*}%
Polynomials $\left\{ T_{n}\right\} $ satisfy the following three-term
recurrence 
\begin{equation}
2xT_{n}(x)=T_{n+1}(x)+T_{n-1}(x),  \label{3Ch}
\end{equation}%
with $T_{0}(x)=1$, $T_{1}(x)=x.$

2. The Chebyshev polynomials of the second kind are defined as polynomials
orthogonal with respect to the semicircle (or Wigner) distribution that has
the following density:%
\begin{equation*}
h(x|3/2,3/2)=%
\begin{cases}
\frac{2}{\pi }\sqrt{1-x^{2}}\text{,} & \text{if }\left\vert x\right\vert <1%
\text{;} \\ 
0\text{,} & \text{if otherwise.}%
\end{cases}%
\end{equation*}%
Again following (\ref{lead}) we deduce that polynomials $K_{n}(x|3/2,3/2)$
are related to the Chebyshev polynomials of the second kind traditionally
denoted by $\left\{ U_{n}\right\} $ and satisfying (\ref{3Ch}) with $%
U_{0}(x)\allowbreak =\allowbreak 1$ and $U_{1}(x)=2x$ in the following way:%
\begin{equation*}
U_{n}(x)=J_{n}(x|3/2,3/2)\frac{2^{2n}n!(n+1)!}{(2n+1)!}.
\end{equation*}

3. The Legendre polynomials $\left\{ P_{n}(x)\right\} $ is the traditional
name for the polynomials that are orthogonal with respect to the measure
with the density equal to $1/2$ on $[-1,1]$ and $0$ otherwise, that is with
respect to $h\left( x|1,1\right) $. It turns out that in this case we have%
\begin{equation*}
P_{n}(x)=J_{n}(x|1,1).
\end{equation*}

4. The Gegenbauer or ultraspherical polynomials $\left\{ C_{n}(x|\lambda
)\right\} _{n\geq 0}$, for $\lambda >-1/2$ are another special case of
Jacobi polynomials, namely we have%
\begin{equation*}
C_{n}(x|\lambda )=\frac{\left( 2\lambda \right) ^{\left( n\right) }}{\left(
\lambda +1/2\right) ^{n}}J_{n}(x|\lambda +1/2,\lambda +1/2).
\end{equation*}

Let us note that we have also the following relationship between even and
odd polynomials orthogonal with respect to symmetric distribution and
non-symmetric distribution.

\begin{lemma}
\label{even}For all $n\geq 0$ and positive $a$ and $b$ we have%
\begin{eqnarray*}
J_{2n}(x|a,a) &=&\frac{n!\left( a+n\right) ^{\left( n\right) }}{(2n)!}%
J_{n}(2x^{2}-1|1/2,a), \\
J_{2n+1}(x|a,a) &=&\frac{n!(a+n)^{(n+1)}}{(2n+1)!}xJ_{n}(2x^{2}-1|3/2,a).
\end{eqnarray*}
\end{lemma}

\begin{proof}
This is common knowledge. See, e.g., Wolfram MathWorld or unnumbered
formulae at the end of page 222 of \cite{KLS}.
\end{proof}

It is well known that for all $n\geq 0$ and complex $x$ and $y$ we have 
\begin{eqnarray*}
\left( x+y\right) ^{\left( n\right) } &=&\sum_{k=0}^{n}\binom{n}{k}\left(
x\right) ^{\left( k\right) }\left( y\right) ^{\left( n-k\right) }, \\
\left( x+y\right) _{\left( n\right) } &=&\sum_{k=0}^{n}\binom{n}{k}\left(
x\right) _{\left( k\right) }\left( y\right) _{\left( n-k\right) }.
\end{eqnarray*}

In order to proceed further, we need the following simple result.`

\begin{lemma}
\label{a-b}For all complex $a$, $b$ and integer $n$ we have 
\begin{eqnarray}
\sum_{j=0}^{n}(-1)^{j}\binom{n}{j}\left( a\right) ^{\left( j\right) }\left(
b+j\right) ^{\left( n-j\right) } &=&\left( b-a\right) ^{(n)},  \label{rozn}
\\
\sum_{j=0}^{n}\left( -1\right) ^{n-j}\binom{n}{j}\left( b+n-1\right)
^{\left( j\right) }\left( a+j\right) ^{\left( n-j\right) } &=&\left(
b-a\right) ^{\left( n\right) },  \label{rozn2} \\
\sum_{j=0}^{n}\left( -1\right) ^{j}\binom{n}{j}\left( b\right) ^{\left(
n-j\right) }\left( a+1-j\right) ^{\left( j\right) } &=&\left( b-a\right)
^{\left( n\right) }.  \label{rozn3}
\end{eqnarray}
\end{lemma}

\begin{proof}
Notice also that $\left( -1\right) ^{n}\left( a+1-n\right) ^{\left( n\right)
}\allowbreak =\allowbreak \left( -a\right) ^{\left( n\right) }$. We have the
following binomial formula true for $\left\vert t\right\vert <1$ 
\begin{equation*}
\sum_{n\geq 0}\frac{t^{n}}{n!}\left( a\right) _{\left( n\right) }=\left(
1+t\right) ^{a},
\end{equation*}%
which after small modification becomes%
\begin{equation*}
\sum_{n\geq 0}\frac{t^{n}}{n!}\left( a\right) ^{\left( n\right) }=\left(
1-t\right) ^{-a}.
\end{equation*}%
Hence, we have in case of (\ref{rozn3}) 
\begin{eqnarray*}
\sum_{n=0}^{\infty }\frac{t^{n}}{n!}\sum_{j=0}^{n}\left( -1\right) ^{j}%
\binom{n}{j}\left( b\right) ^{\left( n-j\right) }\left( a+1-j\right)
^{\left( j\right) } &=& \\
\sum_{j=0}^{\infty }\frac{t^{j}}{j!}(-1)^{j}\left( a+1-j\right) ^{\left(
j\right) }\sum_{n=j}^{\infty }\frac{t^{n-j}}{(n-j)!}\left( b\right) ^{\left(
n-j\right) } &=&\left( 1-t\right) ^{b}\left( 1-t\right) ^{-a}.
\end{eqnarray*}

In case of (\ref{rozn2}), we have%
\begin{eqnarray*}
&&\sum_{j=0}^{n}\left( -1\right) ^{n-j}\binom{n}{j}\left( b+n-1\right)
^{\left( j\right) }\left( a+j\right) ^{\left( n-j\right) } \\
&=&\sum_{s=0}^{n}\left( -1\right) ^{s}\binom{n}{s}\left( b+n-1\right)
^{\left( n-s\right) }\left( a+n-s\right) ^{\left( s\right) }.
\end{eqnarray*}%
We denote $x\allowbreak =\allowbreak a+n-1$, $y\allowbreak =\allowbreak
b+n-1 $. Notice that $y-x\allowbreak =\allowbreak b-a$ . Now we apply \ref%
{rozn3}.

To get (\ref{rozn}) we proceed as follows: 
\begin{gather*}
\sum_{n\geq 0}\frac{t^{n}}{n!}\sum_{j=0}^{n}(-1)^{j}\binom{n}{j}\left(
a\right) ^{\left( j\right) }\left( b+j\right) ^{\left( n-j\right)
}=\sum_{j\geq 0}\frac{(-t)^{j}}{j!}\left( a\right) ^{j}\sum_{n\geq j}\frac{%
t^{n-j}\left( b+j\right) ^{\left( n-j\right) }}{\left( n-j\right) !} \\
\sum_{j\geq 0}\frac{(-t)^{j}}{j!}\left( a\right) ^{\left( j\right)
}(1-t)^{-b-j}=\left( 1-t\right) ^{-b}\sum_{j\geq 0}\frac{(-t/(1-t))^{j}}{j!}%
\left( a\right) ^{\left( j\right) } \\
=\left( 1-t\right) ^{-b}\left( 1+\frac{t}{1-t}\right) ^{a}=\left( 1-t\right)
^{a-b}=\sum_{n\geq 0}\frac{t^{n}}{n!}\left( b-a\right) ^{\left( n\right) }.
\end{gather*}

By the way, from the so called Chu-Vandermonde identity, it follows that 
\begin{equation*}
\sum_{j=0}^{n}(-1)^{j}\frac{\left( a\right) ^{\left( j\right) }}{\left(
b\right) ^{\left( j\right) }}=\frac{\left( b-a\right) ^{\left( n\right) }}{%
\left( b\right) ^{\left( n\right) }},
\end{equation*}%
hence by multiplying both sides of this identity by $\left( b\right)
^{\left( n\right) }$ we get immediately (\ref{rozn}).
\end{proof}

Let us denote by 
\begin{eqnarray}
e_{n,m}(a,b) &=&\binom{n}{m}\left( a+b+n-1\right) ^{\left( m\right)
}(b+m)^{(n-m)}/n!,  \label{cnj} \\
\tilde{e}_{n,m}(a,b) &=&(-1)^{n-m}\frac{n!\left( b+m\right) ^{\left(
n-m\right) }}{(n-m)!\left( a+b+m-1\right) ^{\left( m\right) }\left(
a+b+2m\right) ^{\left( n-m\right) }}  \label{dnj} \\
&=&(-1)^{n-m}\frac{n!\left( b+m\right) ^{\left( n-m\right) }\left(
a+b+2m-1\right) }{(n-m)!\left( a+b+m-1\right) ^{\left( n+1\right) }}.
\label{dnj2}
\end{eqnarray}

The equality of (\ref{dnj}) and (\ref{dnj2}) follows the following trivial
identity 
\begin{equation*}
\left( x\right) ^{\left( j\right) }\left( x+j+1\right) ^{\left( n-j\right) }=%
\frac{\left( x\right) ^{\left( n+1\right) }}{\left( x+j\right) }.
\end{equation*}

\begin{lemma}
\label{odwr}For all $n$, $m\leq n$ and complex $a$, $b$ such that $\func{Re}%
(a),\func{Re}\left( b\right) >0$, we have%
\begin{equation*}
\sum_{k=m}^{n}e_{n,k}(a,b)\tilde{e}_{k,m}\left( a,b\right) =%
\begin{cases}
1\text{,} & \text{if }n=m\text{;} \\ 
0\text{,} & \text{if }0\leq m<n\text{.}%
\end{cases}%
\end{equation*}
\end{lemma}

\begin{proof}
We have%
\begin{gather*}
\sum_{k=m}^{n}e_{n,k}(a,b)\tilde{e}_{k,m}\left( a,b\right) =\sum_{k=m}^{n}%
\frac{\left( a+b+n-1\right) ^{\left( k\right) }}{k!}\frac{(b+k)^{(n-k)}}{%
(n-k)!} \\
\times (-1)^{k-m}\frac{k!\left( b+m\right) ^{\left( k-m\right) }}{%
(k-m)!\left( a+b+m-1\right) ^{\left( m\right) }\left( a+b+2m\right) ^{\left(
k-m\right) }} \\
=\frac{1}{\left( a+b+m-1\right) ^{\left( m\right) }}\sum_{k=m}^{n}\frac{%
\left( a+b+n-1\right) ^{\left( k\right) }(b+k)^{(n-k)}}{(n-k)!} \\
\times (-1)^{k-m}\frac{\left( b+m\right) ^{\left( k-m\right) }}{(k-m)!\left(
a+b+2m\right) ^{\left( k-m\right) }}
\end{gather*}%
\begin{gather*}
=\frac{1}{(n-m)!\left( a+b+m-1\right) ^{\left( m\right) }} \\
\times \sum_{s=0}^{n-m}\left( -1\right) ^{s}\binom{n-m}{s}\frac{\left(
a+b+n-1\right) ^{\left( s+m\right) }(b+m+s)^{(n-m-s)}\left( b+m\right)
^{\left( s\right) }}{\left( a+b+2m\right) ^{\left( s\right) }} \\
=\frac{\left( b+m\right) ^{\left( n-m\right) }\left( a+b+n-1\right) ^{\left(
m\right) }}{(n-m)!\left( a+b+m-1\right) ^{\left( m\right) }}%
\sum_{s=0}^{n-m}\left( -1\right) ^{s}\binom{n-m}{s}\frac{\left(
a+b+n-1+m\right) ^{(s)}}{\left( a+b+2m\right) ^{\left( s\right) }} \\
=\frac{\left( b+m\right) ^{\left( n-m\right) }\left( a+b+n-1\right) ^{\left(
m\right) }}{(n-m)!\left( a+b+m-1\right) ^{\left( m\right) }\left(
a+b+2m\right) ^{\left( n-m\right) }} \\
\times \sum_{s=0}^{n-m}\left( -1\right) ^{s}\binom{n-m}{s}\left(
a+b+n-m-1+2m\right) ^{(s)}\left( a+b+2m+s\right) ^{\left( n-m-s\right) } \\
=\frac{\left( b+m\right) ^{\left( n-m\right) }\left( a+b+n-1\right) ^{\left(
m\right) }}{(n-m)!\left( a+b+m-1\right) ^{\left( m\right) }\left(
a+b+2m\right) ^{\left( n-m\right) }}\left( -n+m+1\right) ^{\left( n-m\right)
}.
\end{gather*}

Now, recall that $\left( a\right) ^{\left( n\right) }\allowbreak
=\allowbreak 0$ when $n>1$ and $-a\allowbreak =\allowbreak 0,1,2,\ldots $ , $%
a$ when $n\allowbreak =\allowbreak 1$ and finally $1$ when $n\allowbreak
=\allowbreak 0.$
\end{proof}

Notice that we have just proved that :%
\begin{eqnarray*}
J_{n}(x|a,b) &=&\sum_{m=0}^{n}e_{n,m}(a,b)(x-1)^{m}/2^{m}, \\
\left( x-1\right) ^{n}/2^{n} &=&\sum_{m=0}^{n}\tilde{e}%
_{n,m}(a,b)J_{m}(x|a,b).
\end{eqnarray*}

\begin{remark}
As mentioned above the assertion of the Lemma \ref{odwr} is proven in \cite%
{IA} (section 4.2) but with a slightly different notation and argumentation.
\end{remark}

As an immediate corollary from this result we have the following observation:

\begin{proposition}
\label{elem}Let $\left\{ J_{n}(x|a,b\right\} $ and $\left\{
J_{n}(x|c,d)\right\} $ be two families of Jacobi polynomials defined by (\ref%
{Ja}) with parameters respectively $a,b$ and $c,d$. Then for all $n\geq 0$
we have:%
\begin{equation}
J_{n}(x|a,b)=\sum_{j=0}^{n}c_{n,j}(a,b;c,d)J_{j}(x|c,d),  \label{conn1}
\end{equation}%
where 
\begin{equation}
c_{n,j}(a,b;c,d)=\sum_{k=j}^{n}e_{n,k}(a,b)\tilde{e}_{k,j}(c,d),
\label{conn2}
\end{equation}%
where $j\allowbreak =\allowbreak 0,\ldots ,n$ and coefficients $e_{n,j}$ and 
$\tilde{e}_{n,j}$ are given by (\ref{cnj}) and (\ref{dnj}).

For all complex $a,b,c,d,e,f$, $n\geq j\geq 0$ we have%
\begin{eqnarray}
\sum_{k=m}^{n}\tilde{e}_{n,k}(a,b)e_{k,m}\left( a,b\right) &=&%
\begin{cases}
1\text{,} & \text{if }n=m\text{;} \\ 
0\text{,} & \text{if }0\leq m<n\text{.}%
\end{cases}
\label{odw2} \\
c_{n,j}\left( a,b;c,d\right) &=&\sum_{k=j}^{n}c_{n,k}\left( a,b;e,f\right)
c_{k,j}\left( e,f;c,d\right) ,  \label{conn3} \\
\sum_{k=j}^{n}c_{n,k}\left( a,b;c,d\right) c_{k,j}\left( c,d;a,b\right) &=&%
\begin{cases}
0, & \text{if }n>j\text{,} \\ 
1, & \text{if }n=j\text{.}%
\end{cases}
\label{odwrr} \\
c_{n,j}\left( a,a;b,b\right) &=&0\text{ if }n-j\text{ is odd.}  \label{conn4}
\\
c_{n,j}(a,b;c,d) &=&\left( -1\right) ^{n-j}c_{n,j}(b,a;d,c).  \label{odwrkol}
\end{eqnarray}
\end{proposition}

\begin{proof}
By Lemma \ref{odwr}, we deduce that $(x-1)^{n}/2^{n}\allowbreak =\allowbreak
\sum_{j=0}^{n}\tilde{e}_{n,j}(c,d)J_{j}(x|c,d)$. Combining this with (\ref%
{J1}) gives (\ref{conn2}). (\ref{conn4}) follows the fact that for $%
a\allowbreak =\allowbreak b$ the distribution $h\left( x|a,a\right) $ is
symmetric consequently the Jacobi polynomials $J_{n}(x|a,a)$ contain only
odd powers of $x$ if $n$ is odd or only even when $n$ is even. (\ref{odwrr}).

In order to get (\ref{odwrkol}), we first recall the following property of
Jacobi polynomials that appear, e.g., in \cite{IA} (4.14) that reads that%
\begin{equation}
(-1)^{n}J_{n}\left( x|a,b\right) =J_{n}\left( -x|b,a\right) .  \label{KK}
\end{equation}

Now we proceed as follows%
\begin{gather*}
J_{n}(-x|b,a)=\left( -1\right)
^{n}J_{n}(x|a,b)=\sum_{j=0}^{n}(-1)^{n}c_{n,j}(a,b;c,d)J_{j}(x|c,d) \\
=\sum_{j=0}^{n}(-1)^{n}\left( -1\right) ^{-j}c_{n,j}(a,b;c,d)\left(
-1\right) ^{j}J_{j}(x|c,d) \\
=\sum_{j=0}^{n}(-1)^{n-j}c_{n,j}(a,b;c,d)J_{j}(-x|d,c).
\end{gather*}

Now notice that $e_{n,j}(a,b)$, $\tilde{e}_{n,j}(a,b)$, $c_{n,j}(a,b;c,d)$
for all $n\geq j\geq 0$ are polynomials in $a$, $b$, $c$, $d$ hence (\ref%
{odw2}), (\ref{conn2}), (\ref{conn3}), (\ref{conn4}), (\ref{odwrr}), (\ref%
{odwrkol}) can be extended to all complex numbers.
\end{proof}

\begin{remark}
In order to understand better the assertions of the above-mentioned
Proposition let us define a set of lower triangular matrices $%
E(a,b)\allowbreak =\allowbreak \lbrack e_{n,k}(a,b)]$, $\tilde{E}%
(c,d)\allowbreak =\allowbreak \lbrack \tilde{e}_{n,k}(c,d)]$, $%
C(a,b;c,d)\allowbreak =\allowbreak \lbrack c_{n,j}(a,b;c,d)]$. Then we see
that the assertion of the Lemma (\ref{odwr}) and the formulae (\ref{odw2}), (%
\ref{conn2}) and (\ref{odwrr}) mean in terms of these matrices the following
respective identities:%
\begin{eqnarray}
E^{-1}(a,b) &=&\tilde{E}(a,b),  \label{M1} \\
C(a,b;c,d) &=&E(a,b)E^{-1}(c,d),  \label{M2} \\
C^{-1}(a,b;c,d) &=&C(c,d;a,b).  \label{M3}
\end{eqnarray}
\end{remark}

\begin{proposition}
\label{ccon}For all $n\geq 0$, $j\allowbreak =\allowbreak 0,\ldots n$ and $a 
$, $b\geq 0$ we have. Following directly (\ref{cnj}) and (\ref{dnj}) we
arrive at%
\begin{eqnarray}
e_{n,j}(b,b) &=&\binom{n}{j}(2b+n-1)^{(j)}(b+j)^{(n-j)}/n!,  \label{ebb} \\
\tilde{e}_{n,j}(b,b) &=&(b,b)=(-1)^{n-j}\frac{n!(b+j)^{(n-j)}(2b+2j-1)}{%
(n-j)!(2b+j-1)^{(n+1)}},  \label{oebb} \\
e_{n,j}(a,1/2) &=&\frac{\left( a-1/2+n\right) ^{\left( j\right) }\left(
1/2\right) ^{\left( n\right) }}{j!\left( n-j\right) !\left( 1/2\right)
^{\left( j\right) }},  \label{ea1/2} \\
\tilde{e}_{n,j}(a,1/2) &=&\left( -1\right) ^{n-j}\frac{\left(
a-1/2+2j\right) \left( 1/2+j\right) ^{\left( n-j\right) }j!}{\left(
n-j\right) !\left( a-1/2+j\right) ^{\left( n+1\right) }},  \label{oea1/2} \\
e_{n,j}(a,3/2) &=&\frac{\left( a+1/2+n\right) ^{\left( j\right) }\left(
3/2\right) ^{\left( n\right) }}{j!\left( n-j\right) !\left( 3/2\right)
^{\left( j\right) }},  \label{ea3/2} \\
\tilde{e}_{n,j}(a,3/2) &=&\left( -1\right) ^{n-j}\frac{\left(
a+1/2+2j\right) \left( 3/2+j\right) ^{\left( n-j\right) }j!}{\left(
n-j\right) !\left( a+1/2+j\right) ^{\left( n+1\right) }}.  \label{oea3/2}
\end{eqnarray}

\begin{gather}
c_{n,j}(a,1/2;b,1/2)=\left( -1\right) ^{n-j}c_{n,j}(1/2,a;1/2,b)
\label{a1/2} \\
=\left( -1\right) ^{n-j}\frac{\left( 1/2\right) ^{\left( n\right) }\left(
a-b\right) ^{\left( n-j\right) }\left( a-1/2+n\right) ^{\left( j\right)
}\left( b-1/2+2j\right) }{\left( n-j\right) !\left( 1/2\right) ^{\left(
j\right) }\left( b+1/2+2j\right) ^{\left( n-j\right) }\left( b-1/2+j\right)
^{\left( j+1\right) }},  \notag \\
c_{n,j}(a,3/2;b,3/2)=\left( -1\right) ^{n-j}c_{n,j}(3/2,a;3/2,b)
\label{a3/2} \\
=\left( -1\right) ^{n-j}\frac{\left( 3/2\right) ^{\left( n\right) }\left(
a-b\right) ^{\left( n-j\right) }\left( a+1/2+n\right) ^{\left( j\right)
}(b+1/2+2j)}{\left( n-j\right) !\left( 3/2\right) ^{\left( j\right) }\left(
b+3/2+2j\right) ^{\left( n-j\right) }\left( b+1/2+j\right) ^{\left(
j+1\right) }}  \notag
\end{gather}%
\begin{gather}
c_{n,j}(a,b;b,b)=(-1)^{n-j}c_{n,j}(b,a;b,b)  \label{ab} \\
=(-1)^{n-j}\frac{(b+j)^{(n-j)}(a-b)^{(n-j)}(a+b+n-1)^{(j)}(2b+2j-1)}{%
(n-j)!(2b+j-1)^{(n+1)}}  \notag \\
c_{n,j}(b,b;a,b)=(-1)^{n-j}c_{n,j}(b,b;b,a)  \label{ba} \\
=(-1)^{n-j}\frac{(b+j)^{(n-j)}(2b+n-1)^{(j)}(b-a)^{(n-j)}\left(
a+b+2j-1\right) }{(n-j)!(a+b+j-1)^{(n+1)}},  \notag
\end{gather}

\begin{gather}
c_{n,j}\left( a,a;b,b\right) =  \notag \\
\begin{cases}
0\text{,} & \text{ if }n-j\text{ is odd;} \\ 
\frac{(2b+2j-1)\left( 2a+n-1\right) ^{\left( j\right) }\left( a-b\right)
^{\left( \left( n-j\right) /2\right) }\left( b+j\right) ^{\left( \left(
n-j\right) /2\right) }\left( a+\left( n+j\right) /2\right) ^{\left( \left(
n-j\right) /2\right) }}{\left( \left( n-j\right) /2\right) !\left(
2b+j-1\right) ^{\left( n+1\right) }}\text{,} & \text{if }n-j\text{ is even.}%
\end{cases}%
\text{ }  \label{aabb}
\end{gather}
\end{proposition}

\begin{proof}
Is moved to Section \ref{Dow}.
\end{proof}

\begin{remark}
The formulae (\ref{ab}), (\ref{ba}) and (\ref{aabb}) were obtained by R.
Askey in 1975 by other methods based on the properties of hypergeometric
functions (see chapter 7 of \cite{Ask75}).
\end{remark}

Formulae (\ref{ab}), (\ref{ba}), (\ref{aabb}), can be the source of very
large number of identities involving Pochhammer symbol. They could be based
on the identities true for all $n\geq j\geq 0$ and positive $a$ and $b.$%
\begin{eqnarray*}
e_{n,j}(a,b) &=&\sum_{k=j}^{n}c_{n,k}(a,b;b,b)e_{k,j}(b,b), \\
\tilde{e}_{n,j}(b,b) &=&\sum_{k=j}^{n}\grave{e}_{n,k}(a,b)c_{k,j}(a,b;b,b),
\\
c_{n,j}(a,a;b,b) &=&\sum_{k=j}^{n}c_{n,k}(a,a,a,b)c_{k,j}(a,b;b,b), \\
c_{n,j}(a,a;a,b) &=&\sum_{k=j}^{n}c_{n,k}(a,a;b,b)c_{k,j}(b,b;a,b)
\end{eqnarray*}%
and so on. The above-mentioned identities are based on the observations (\ref%
{M1}), (\ref{M2}), (\ref{conn3}). In the corollary below we present only a
sample of such identities.

\begin{corollary}
\label{Upr}The following identities, involving Pochhammer symbols of two
variables, can be obtained from (\ref{ab}), (\ref{ba}) and (\ref{01}) of
course adapted to current setting of Jacobi polynomials. We have for all $%
n>0 $ and complex $x$, $y$, $a$, $b$ having positive real parts%
\begin{gather}
\sum_{j=0}^{n}\left( -1\right) ^{n-j}\binom{n}{j}\left( x+y+n-1\right)
^{\left( n-j\right) }\left( x+n-j\right) ^{\left( j\right) }\left(
2y+n-j\right) ^{\left( j\right) }\left( y\right) ^{\left( n-j\right) }
\label{x-y} \\
=\left( x-y\right) ^{\left( n\right) }\left( y\right) ^{\left( n\right) }, 
\notag \\
\sum_{j=0}^{n}\left( -1\right) ^{j}\binom{n}{j}\left( 2y+n-1\right) ^{\left(
j\right) }\left( y+j\right) ^{\left( n-j\right) }\left( x\right) ^{\left(
j\right) }\left( x+y+j\right) ^{\left( n-j\right) }  \label{y-x} \\
=\left( y-x\right) ^{\left( n\right) }\left( y\right) ^{\left( n\right) }, 
\notag \\
\sum_{j=0}^{n}\binom{n}{j}\frac{\left( 2b+n-1\right) ^{\left( j\right)
}\left( b-a\right) ^{\left( n-j\right) }\left( a-b\right) ^{\left( j\right)
}(a+b+2j-1)}{\left( a+b+j-1\right) ^{\left( n+1\right) }\left( 2b\right)
^{\left( j\right) }}=0,  \label{001} \\
\sum_{j=0}^{n}\binom{n}{j}\frac{\left( a+b+n-1\right) ^{\left( j\right)
}\left( a-b\right) ^{\left( n-j\right) }\left( b-a\right) ^{\left( j\right)
}(2b+2j-1)}{\left( 2b+j-1\right) ^{\left( n+1\right) }\left( a+b\right)
^{\left( j\right) }}=0,  \label{002}
\end{gather}

\begin{gather}
\sum_{j=0}^{2n}\left( -1\right) ^{j}\binom{2n}{j}\left( x+j\right) ^{\left(
2n-j\right) }\left( 2x+2n-1\right) ^{\left( j\right) }\left( y\right)
^{\left( j\right) }\left( 2y+j\right) ^{\left( 2n-j\right) }  \label{dd1} \\
=\frac{\left( 2n\right) !}{n!}\left( x-y\right) ^{\left( n\right) }\left(
y\right) ^{\left( n\right) }\left( x+n\right) ^{\left( n\right) },  \notag \\
\sum_{j=0}^{2n+1}\left( -1\right) ^{j}\binom{2n+1}{j}\left( 2x+2n\right)
^{\left( j\right) }\left( x+j\right) ^{\left( 2n+1-j\right) }\left( y\right)
^{\left( j\right) }\left( 2y+j\right) ^{\left( 2n+1-j\right) }=0,
\label{dd2}
\end{gather}

\begin{eqnarray}
\sum_{j=0}^{n}\binom{n}{j}\frac{\left( x-y\right) ^{\left( n-j\right)
}\left( y-x\right) ^{\left( j\right) }\left( x+y+n-1\right) ^{\left(
j\right) }}{\left( x+y\right) ^{\left( j\right) }\left( 2y+j-1\right)
^{\left( j\right) }\left( 2y+2j\right) ^{\left( n-j\right) }} &=&0,
\label{aaababaa} \\
\sum_{j=0}^{n}\binom{n}{j}\frac{\left( y-x\right) ^{\left( j\right) }\left(
x-y\right) ^{\left( n-j\right) }\left( x+n-1/2\right) ^{\left( j\right) }}{%
\left( x+1/2\right) ^{\left( j\right) }\left( y+j-1/2\right) ^{\left(
j\right) }\left( y+2j+1/2\right) ^{\left( n-j\right) }} &=&0.
\label{aabbbbaa}
\end{eqnarray}
\end{corollary}

\begin{proof}
Is shifted to Section \ref{Dow}.
\end{proof}

Notice that identities (\ref{x-y}), (\ref{y-x}), (\ref{dd1}), (\ref{dd2})
are valid for all complex $x$ and $y$.

More properties of orthogonal polynomials one can read in \cite{Chih79}, 
\cite{Sim98} or \cite{Sim05}. Take now 
\begin{equation*}
\allowbreak d\alpha \left( x\right) \allowbreak =\allowbreak
(x+1)^{a-1}\left( 1-x\right) ^{b-1}/B\left( a,b\right)
\end{equation*}%
and 
\begin{equation*}
d\beta \left( x\right) \allowbreak =\allowbreak (x+1)^{c-1}\left( 1-x\right)
^{d-1}/B\left( c,d\right) ,
\end{equation*}%
we see that 
\begin{equation}
\frac{d\beta }{d\alpha }\left( x\right) =2^{-(c-a)-(d-b)}\frac{B(a,b)}{%
B\left( c,d\right) }\left( x+1\right) ^{c-a}\left( 1-x\right) ^{d-b}.
\label{stos}
\end{equation}%
Hence it is a beta density if $c-a,d-b>-1$ and 
\begin{equation*}
\int \left( \frac{d\beta }{d\alpha }(x)\right) ^{2}d\alpha \left( x\right)
\allowbreak =\allowbreak \int \left( x+1\right)
^{2c-a-1}(1-x)^{2d-b-1}dx<\infty .
\end{equation*}%
The last integral is finite if $2c-a,2d-b>0.$

Moreover, notice also that%
\begin{equation*}
c_{n,0}(c,d;a,b)\allowbreak =\allowbreak \sum_{m=0}^{n}\frac{\left(
c+d+n-1\right) ^{\left( m\right) }(d+m)^{(n-m)}}{m!(n-m)!}(-1)^{m}\frac{%
\left( b\right) ^{\left( m\right) }}{(a+b)^{\left( m\right) }},
\end{equation*}%
and 
\begin{equation*}
\hat{a}_{n}(a,b)\allowbreak =\allowbreak \frac{\left( a\right) ^{\left(
n\right) }\left( b\right) ^{\left( n\right) }}{n!(a+b+2n-1)(a+b)^{\left(
n-1\right) }}.
\end{equation*}

Notice also that following (\ref{stos}) 
\begin{equation*}
\frac{h(x|c,d)}{h\left( x|a,b\right) }=\frac{2B(a,b)B\left(
c-a+1,d-b+1\right) }{B\left( c,d\right) }h\left( x|c-a+1,d-b+1\right) .
\end{equation*}%
Hence, we have the following two infinite, convergent in mean-squares,
expansions :

\begin{theorem}
For $2c>a$ and $2d>b$ and $x\in (-1,1)$ we have%
\begin{eqnarray*}
h(x|c,d) &=&h(x|a,b)\sum_{n\geq 0}c_{n,0}(a,b;c,d)J_{n}(x|a,b)/\hat{a}%
_{n}(a,b), \\
\frac{h(x|c,d)}{h\left( x|a,b\right) } &=&\frac{B(a,b)B\left(
c-a+1,d-b+1\right) }{B\left( c,d\right) } \\
&&\times \sum_{n\geq 0}(2n+1)c_{n,0}\left( 1,1,c-a+1,d-b+1\right)
J_{n}(x|1,1).
\end{eqnarray*}
\end{theorem}

\begin{proof}
Follows directly Theorem \ref{podst} and the formulae concerning Beta
distribution.
\end{proof}

\section{Notation and basic definitions used in $q-$series\label{q-series}}

$q$ is a parameter. We will assume that $-1<q\leq 1$ unless otherwise
stated. The case $q\allowbreak =\allowbreak 1$ may not always be considered
directly, but sometimes as left-hand side limit ( i.e.,$q\longrightarrow
1^{-}$). We will point out these cases.

We will use traditional notations of the $q-$series theory i.e.,%
\begin{equation*}
\left[ 0\right] _{q}\allowbreak=\allowbreak0,~\left[ n\right]
_{q}\allowbreak=\allowbreak1+q+\ldots+q^{n-1}\allowbreak,\left[ n\right]
_{q}!\allowbreak=\allowbreak\prod_{j=1}^{n}\left[ j\right] _{q},\text{with }%
\left[ 0\right] _{q}!\allowbreak=1\text{,}
\end{equation*}%
\begin{equation*}
\QATOPD{[}{]}{n}{k}_{q}=\left\{ 
\begin{array}{ccc}
\frac{\left[ n\right] _{q}!}{\left[ n-k\right] _{q}!\left[ k\right] _{q}!} & 
, & n\geq k\geq0 \\ 
0 & , & \text{otherwise}%
\end{array}
\right. \text{.}
\end{equation*}
$\binom{n}{k}$ will denote the ordinary, well known binomial coefficient.

It is useful to use the so-called $q-$Pochhammer symbol for $n\geq1:$%
\begin{equation*}
\left( a|q\right) _{n}=\prod_{j=0}^{n-1}\left( 1-aq^{j}\right) ,~~\left(
a_{1},a_{2},\ldots,a_{k}|q\right)
_{n}\allowbreak=\allowbreak\prod_{j=1}^{k}\left( a_{j}|q\right) _{n}\text{,}
\end{equation*}
with $\left( a|q\right) _{0}\allowbreak=\allowbreak1$.

Often $\left( a|q\right) _{n}$ as well as $\left( a_{1},a_{2},\ldots
,a_{k}|q\right) _{n}$ will be abbreviated to $\left( a\right) _{n}$ and 
\newline
$\left( a_{1},a_{2},\ldots ,a_{k}\right) _{n}$, if it will not cause
misunderstanding. In this paper most often $\left( a|q\right) _{n}$ will be
abbreviated to $\left( a\right) _{n}$.

We will also use the following symbol $\left\lfloor n\right\rfloor $ to
denote the largest integer not exceeding $n$.

It is worth to mention the following $4$ formulae, that are well known.
Namely, the following formulae are true for $\left\vert t\right\vert <1$, $%
\left\vert q\right\vert <1$ (derived already by Euler, see \cite{Andrews1999}
Corollary 10.2.2 or \cite{KLS}(Subsections 1.8 1.14)) 
\begin{align}
\frac{1}{(t)_{\infty }}\allowbreak & =\allowbreak \sum_{k\geq 0}\frac{t^{k}}{%
(q)_{k}}\text{, }\frac{1}{(t)_{n+1}}=\sum_{j\geq 0}\QATOPD[ ] {n+j}{j}%
_{q}t^{j}\text{,}  \label{binT} \\
(t)_{\infty }\allowbreak & =\allowbreak \sum_{k\geq 0}(-1)^{k}q^{\binom{k}{2}%
}\frac{t^{k}}{(q)_{k}}\text{, }\left( t\right) _{n}=\sum_{j=0}^{n}\QATOPD[ ]
{n}{j}_{q}q^{\binom{j}{2}}(-t)^{j}\text{.}  \label{obinT}
\end{align}

In particular, we have (after setting $t\allowbreak =\allowbreak 1$) for
finite $n>0$ and all complex $q$%
\begin{equation*}
0=\sum_{j=0}^{n}\QATOPD[ ] {n}{j}_{q}q^{\binom{j}{2}}(-1)^{j}.
\end{equation*}%
If we pass with $n$ to infinity then for all $\left\vert q\right\vert <1$ we
have%
\begin{equation*}
0=\sum_{j\geq 0}(-1)^{j}q^{\binom{j}{2}}/\left( q\right) _{j}.
\end{equation*}%
It is easy to notice that 
\begin{equation*}
\left( q\right) _{n}=\left( 1-q\right) ^{n}\left[ n\right] _{q}!
\end{equation*}%
and that%
\begin{equation*}
\QATOPD[ ] {n}{k}_{q}\allowbreak =\allowbreak \allowbreak \left\{ 
\begin{array}{ccc}
\frac{\left( q\right) _{n}}{\left( q\right) _{n-k}\left( q\right) _{k}} & ,
& n\geq k\geq 0 \\ 
0 & , & \text{otherwise}%
\end{array}%
\right. \text{.}
\end{equation*}%
\newline
The above-mentioned formula is just an example, where direct setting $%
q\allowbreak =\allowbreak 1$ is senseless, however, the passage to the limit 
$q\longrightarrow 1^{-}$ makes sense.

Notice that, in particular we get 
\begin{equation}
\left[ n\right] _{1}\allowbreak =\allowbreak n,~\left[ n\right]
_{1}!\allowbreak =\allowbreak n!,~\QATOPD[ ] {n}{k}_{1}\allowbreak
=\allowbreak \binom{n}{k},~(a)_{1}\allowbreak =\allowbreak 1-a,~\left(
a|1\right) _{n}\allowbreak =\allowbreak \left( 1-a\right) ^{n}  \label{q1}
\end{equation}%
and 
\begin{equation}
\left[ n\right] _{0}\allowbreak =\allowbreak \left\{ 
\begin{array}{ccc}
1 & \text{if} & n\geq 1 \\ 
0 & \text{if} & n=0%
\end{array}%
\right. ,~\left[ n\right] _{0}!\allowbreak =\allowbreak 1,~\QATOPD[ ] {n}{k}%
_{0}\allowbreak =\allowbreak 1,~\left( a|0\right) _{n}\allowbreak
=\allowbreak \left\{ 
\begin{array}{ccc}
1 & \text{if} & n=0 \\ 
1-a & \text{if} & n\geq 1%
\end{array}%
\right. .  \label{q2}
\end{equation}

The symbol $i$ will denote the imaginary unit, unless otherwise clearly
stated. Let us define also:%
\begin{align}
v(x|a)\allowbreak & =\allowbreak 1-2ax+a^{2},  \label{vxa} \\
l(x|a)& =(1+a)^{2}-4x^{2}a,  \label{lsa} \\
w(x,y|a)& =(1-a^{2})^{2}-4xya(1+a^{2})+4a^{2}(x^{2}+y^{2}).  \label{wxya}
\end{align}%
Notice that, we have

\begin{align}
\left( ae^{i\theta },ae^{-i\theta }\right) _{n}& =\prod_{k=0}^{n}v\left(
x|aq^{k}\right) \text{,}  \label{rozklv} \\
\left( te^{i\left( \theta +\phi \right) },te^{i\left( \theta -\phi \right)
},te^{-i\left( \theta -\phi \right) },te^{-i\left( \theta +\phi \right)
}\right) _{n}& =\prod_{k=0}^{n}w\left( x,y|tq^{k}\right) \text{,}
\label{rozklw} \\
\left( ae^{2i\theta },ae^{-2i\theta }\right) _{n}& =\prod_{k=0}^{n}l\left(
x|aq^{k}\right) \text{,}  \label{rozkll}
\end{align}%
where, and, as usually in the $q-$series theory, $x\allowbreak =\allowbreak
\cos \theta $ and $y=\cos \phi $.

In the sequel we will often use the following easy to justify identities
taken almost directly from \cite{KLS}(Sections 1.8, 1.9, 1.10)

\begin{eqnarray}
\left( a\right) _{n+k} &=&\left( a\right) _{n}\left( aq^{n}\right) _{k},
\label{s1} \\
\frac{\left( aq^{n}\right) _{k}}{\left( aq^{k}\right) _{n}} &=&\frac{\left(
a\right) _{k}}{\left( a\right) _{n}},  \label{s2} \\
\left( a^{2}|q^{2}\right) _{\infty } &=&\left( a\right) _{\infty }\left(
-a\right) _{\infty },  \label{s3} \\
\left( a\right) _{\infty } &=&\left( a|q^{2}\right) _{\infty }\left(
aq|q^{2}\right) _{\infty }.  \label{s4}
\end{eqnarray}

In order to simplify some expressions, we will often use the following easy
to justify formulae true for $n\geq k\geq 0:$ 
\begin{eqnarray}
\left( aq^{k-1}\right) _{k}\left( aq^{2k}\right) _{n-k}\allowbreak
&=&\allowbreak \left( aq^{k-1}\right) _{n}\frac{\left( 1-aq^{n+k-1}\right) }{%
\left( 1-aq^{2k-1}\right) },  \label{knk1} \\
(a)_{k}\left( aq^{n+k-1}\right) _{n-k}\allowbreak &=&\allowbreak \frac{%
\left( a\right) _{2n-1}}{\left( aq^{k}\right) _{n-1}}.  \label{knk2}
\end{eqnarray}

\section{Polynomial identities}

As presented in the introduction, to prove the identity, all we need are the
related pairs of orthogonal polynomials and the sets of CC between them.
That is the rest of the paper is organized in the following way. We will
recall the pair of families of orthogonal polynomials, indicated where one
can find their definition and basic properties and the sets of CC, if they
are known. If not, we will derive them and then present the two identities.
That is the rest of the paper is organized in the following way. We will
recall the pair of families of orthogonal polynomials, indicated where one
can find their definition and basic properties and the sets of CC if they
are known. If not, we will derive them and then present the two identities.

One has to point out that since $q-$Pochhammer symbol is a polynomial in
several variables, hence the identity where it appears, is at most a
rational function of its variables, consequently the identity, being
primarily true for reals or conjugate pairs of complex variables can be
extended to all complex numbers.

As remarked in, say \cite{Szab-bAW},\cite{Szab22}, the considered families
of polynomials are orthogonal with respect to measures supported either on $%
[-1,1]$ or on $S(q)\allowbreak =\allowbreak \lbrack -\frac{2}{\sqrt{1-q}},%
\frac{2}{\sqrt{1-q}})$, (if the parameter $q$ is fixed). In this paper we
will consider only the first case i.e. all measures that makes our
polynomials orthogonal will be supported on $[-1,1]$.

As mentioned in the Introduction, there are at least $7$ families of
orthogonal polynomials (Chebyshev, $q-$Hermite, big $q-$Hermite, Rogers,
Al-Salam--Chihara, continuous dual $q-$Hahn., Askey--Wilson using
terminology of \cite{KLS}), that can be considered as belonging directly to
AW scheme and having absolutely continuous measure which make them
orthogonal. Thus, theoretically we have $21\allowbreak =\allowbreak \binom{7%
}{2}$ pairs and consequently $21$ and possible identities. However, not all
of them are new and interesting. For example, the pair of $q-$Hermite and
big $q-$Hermite polynomials, the pair of big $q-$Hermite and
Al-Salam--Chihara polynomials or the pair of $q-$Hahn and Al-Salam--Chihara
polynomials produce trivial identities that can be derived directly from the
binomial theorem (\ref{obinT}) with $t\allowbreak =\allowbreak 1$. As the
result, we will analyze $8$ pairs of polynomials from AW scheme.

One has to observe that in some cases we obtain the well known identities
after applying some relatively simple simplifications. This shows that our
idea of seeking useful identities in an organized way is just.

This section will be divided on subsections named after the names of the
polynomials forming a chosen pair $\left\{ \alpha _{n},\beta _{n}\right\} $
of families of polynomials. We will start each subsection by the reference
to the literature where the given pair of polynomials is present, then we
will present the mutual expansions of each member of a pair with respect to
the other thus providing the two sets of connection coefficients. Then we
give sequences of numbers $\left\{ \hat{\alpha}_{n},\hat{\beta}_{n}\right\} $%
.

Some of these families of polynomials have traditional names and symbols
denoting them. Let us mention these traditional notations and terminology.

Chebyshev of the second kind are traditionally denoted by $U_{n}(x)$. The $%
q- $Hermite (proper name is continuous $q-$Hermite) polynomials are
traditionally denoted as $\left\{ h_{n}(x|q)\right\} $ (compare \cite{KLS}).
The Al-Salam-Chihara (briefly ASC) polynomials are denoted as $\left\{
Q_{n}(x|a,b,q)\right\} $, (compare \cite{KLS}). The Rogers or $q-$%
ultraspherical polynomials are denoted $\left\{ C_{n}(x|\beta ,q)\right\} $,
(see \cite{KLS}).

\subsection{$q-$Hermite and Chebyshev of the second kind.}

These families of polynomials are described, e.g., in \cite{KLS}(Section
14.26, $q-$Hermite), or \cite{Szab-bAW}(Section 3.1, $q-$Hermite and Section
2.2 ). We know also the CC between these families since they are based on
the famous formula for "change of basis" in $q-$Hermite polynomials
presented, e.g., \cite{Szab-bAW}(formula 3.8). Consequently, we have:%
\begin{eqnarray*}
U_{n}(x) &=&\sum_{j=0}^{\left\lfloor n/2\right\rfloor }(-1)^{j}q^{\binom{j+1%
}{2}}\QATOPD[ ] {n-j}{j}_{q}h_{n-2j}(x|q), \\
h_{n}\left( y|q\right) &=&\sum_{k=0}^{\left\lfloor n/2\right\rfloor }\frac{%
q^{k}-q^{n-k+1}}{1-q^{n-k+1}}\QATOPD[ ] {n}{k}_{q}U_{n-2k}\left( y\right) .
\end{eqnarray*}%
Hence, we can read coefficients $c_{n,j}$ and $\bar{c}_{n,j}$. In particular
we have 
\begin{equation*}
c_{2k,0}\allowbreak =\allowbreak \left( -1\right) ^{k}q^{\binom{k+1}{2}}~%
\text{and~}\bar{c}_{2k,0}\allowbreak =\allowbreak \QATOPD[ ] {2k}{k}_{q}%
\frac{q^{k}-q^{k+1}}{1-q^{k+1}}
\end{equation*}%
for $j\allowbreak =\allowbreak 2k$ and $0$ otherwise. We also have $\hat{U}%
_{n}\allowbreak =\allowbreak 1$ and $\hat{h}_{n}(q)\allowbreak =\allowbreak
\left( q\right) _{n}$. Consequently, the following result follows from these
two expansions:

\begin{theorem}
i) For all $m\geq 1$ and $\left\vert q\right\vert <1$%
\begin{gather*}
\sum_{j=0}^{m}(-1)^{j}q^{\binom{j}{2}}\QATOPD[ ] {m}{m-j}_{q}\QATOPD[ ] {2m-j%
}{m}_{q}\frac{1}{(1-q^{m-j+1})}=0, \\
0=\sum_{k=0}^{m}(-1)^{k}q^{\binom{k}{2}}\QATOPD[ ] {2m}{m-k}_{q}\frac{%
1-q^{2k+1}}{1-q^{m+k+1}}.
\end{gather*}

\begin{eqnarray}
\left( q\right) _{\infty }\prod_{j=1}^{\infty }l(x|q^{j}) &=&\sum_{j\geq
0}(-1)^{j}q^{\binom{j+1}{2}}U_{2j}(x),  \label{inU} \\
\frac{1}{\left( q\right) _{\infty }\prod_{j=1}^{\infty }l(x|q^{j})}
&=&\sum_{j\geq 0}\frac{(1-q)q^{j}}{\left( q\right) _{j}^{2}(1-q^{j+1})}%
h_{2j}(x|q).  \label{nah}
\end{eqnarray}
\end{theorem}

\begin{proof}
Let us recall (following say \cite{IA} and/or \cite{Andrews1999}) that 
\begin{equation*}
\int_{-1}^{1}U_{n}\left( x\right) U_{m}\left( x\right) \frac{2\sqrt{1-x^{2}}%
}{\pi }dx=%
\begin{cases}
1, & \text{if }m=n\text{;} \\ 
0, & \text{if }m\neq n\text{,}%
\end{cases}%
\end{equation*}%
and also that 
\begin{equation}
\int_{-1}^{1}h_{n}\left( x|q\right) h_{m}\left( x|q\right) f_{h}\left(
x|q\right) dx=%
\begin{cases}
\left( q\right) _{n}, & \text{if }m=n\text{;} \\ 
0, & \text{if }m\neq n\text{.}%
\end{cases}
\label{inth^2}
\end{equation}%
where we denoted 
\begin{equation}
f_{h}\left( x|q\right) =\frac{2\left( q\right) _{\infty }\sqrt{1-x^{2}}}{\pi 
}\prod_{k=1}^{\infty }l\left( x|q^{k}\right) ,  \label{fh}
\end{equation}%
where $l$ is defined by (\ref{lsa}).
\end{proof}

\begin{remark}
In \cite{Szablowski2009} there are presented many particular cases of the
expansion (\ref{inU}). Hence, let us present a particular case of the
expansion (\ref{nah}). Namely, let us take $x\allowbreak =\allowbreak 0$,
then we notice that%
\begin{equation*}
h_{2j}(0|q)\allowbreak =(-1)^{j}\prod_{k=0}^{j-1}(1-q^{1+2k})=(-1)^{j}\left(
q|q^{2}\right) _{j}.
\end{equation*}%
Besides, we can easily notice that 
\begin{equation*}
\prod_{k=1}^{\infty }l\left( 0|q^{k}\right) \allowbreak \allowbreak
=\allowbreak \prod_{k=1}^{\infty }(1+q^{k})^{2}\allowbreak =\allowbreak
\left( -q\right) _{\infty }^{2}.
\end{equation*}
Thus, after cancelling out $2/\pi $ on both sides and dividing both sides by 
$\left( q\right) _{\infty }(-q)_{\infty }^{2}$ and finally noticing that%
\begin{equation*}
\left( q\right) _{\infty }(-q)_{\infty }\allowbreak =\allowbreak \left(
q^{2}|q^{2}\right) _{\infty },
\end{equation*}
we get the following infinite expansion:%
\begin{equation*}
\frac{1}{\left( q^{2}|q^{2}\right) _{\infty }(-q)_{\infty }}=\sum_{j\geq 0}%
\frac{(-q)^{j}(1-q)}{\left( q\right) _{j}^{2}(1-q^{j+1})}\left(
q|q^{2}\right) _{j}.
\end{equation*}
\end{remark}

\begin{remark}
Following Proposition 7.1 of \cite{Kac2002} one can notice that the
so-called Galois number $G_{n}(q)$ (the total number of subspaces of the
vector $\mathbb{F}_{q}^{n}$ over the finite field $\mathbb{F}_{q}$, of
course, for $q$ being a prime number) is equal to $h_{n}(1|q)$. This is so
since three-term recurrence satisfied by the $q-$ Hermite polynomials is%
\begin{equation*}
h_{n}(x|q)=2xh_{n-1}(x|q)+(q^{n-1}-1)h_{n-2}(x|q),
\end{equation*}%
with $h_{0}(x|q)\allowbreak =\allowbreak 1$, $h_{1}(x|q)\allowbreak
=\allowbreak 2x$. On the way let us notice that $l(1|a)\allowbreak
=\allowbreak (1-a)^{2}$, hence we have for all complex $\left\vert
q\right\vert <1$%
\begin{equation*}
\frac{1}{\left( q\right) _{\infty }^{3}}=\sum_{j\geq 0}\frac{(1-q)q^{j}}{%
\left( q\right) _{j}^{2}(1-q^{j+1})}G_{2j}(q).
\end{equation*}
\end{remark}

\subsection{$q-$ultraspherical (Rogers) and $q-$ultraspherical (Rogers) with
different parameters}

$q-$ultraspherical (Rogers) polynomials are more properly called continuous 
\newline
$q-$ultraspherical polynomials and are defined and described in \cite{KLS}%
(Section 14.10.1) and in more detail in \cite{Szab-bAW}(Section 4.3). There
also is presented a formula 4.15 (see also \cite{IA},(13.3.1)) (dating back
to Rogers in the end of 19th century) giving connection coefficients between
two sets of Rogers polynomials with different values of the parameter $\beta
.$

Namely, we have

\begin{equation}
C_{n}\left( x|\gamma ,q\right) \allowbreak =\allowbreak
\sum_{k=0}^{\left\lfloor n/2\right\rfloor }\frac{\beta ^{k}\left( \gamma
/\beta \right) _{k}\left( \gamma \right) _{n-k}\left( 1-\beta
q^{n-2k}\right) }{(q)_{k}\left( \beta q\right) _{n-k}\left( 1-\beta \right) }%
C_{n-2k}\left( x|\beta ,q\right) .  \label{CnaC}
\end{equation}

Again, we can read coefficients $c_{n,j}$ and $\bar{c}_{n,j}$ from (\ref%
{CnaC}). Thus, we have coefficients 
\begin{equation*}
c_{0,j}\allowbreak =\allowbreak \frac{\beta ^{k}\left( \gamma /\beta \right)
_{k}\left( \gamma \right) _{k}}{\left( q\right) _{k}\left( \beta q\right)
_{k}(1-\beta )}~~\text{and~~}\hat{c}_{0,\allowbreak j}\allowbreak
=\allowbreak \frac{\gamma ^{k}\left( \beta /\gamma \right) _{k}\left( \beta
\right) _{k}}{\left( q\right) _{k}\left( \gamma q\right) _{k}(1-\gamma )}
\end{equation*}
for $j\allowbreak =\allowbreak 2k$ and $0$ otherwise. Recall e.g. \cite{IA}%
(13.2.4) or \cite{Szab22}(4.13,4.14) and let us modify slightly $f_{C}$ (by
multiplying by $(1-\beta )$), so that $f_{C}$ integrates to $1$. We get 
\begin{equation}
\int_{-1}^{1}C_{n}\left( x|\beta ,q\right) C_{m}\left( x|\beta ,q\right)
f_{C}\left( x|\beta ,q\right) dx=\left\{ 
\begin{array}{ccc}
0 & \text{if} & m\neq n \\ 
\frac{\left( \beta ^{2}\right) _{n}(1-\beta )}{\left( 1-\beta q^{n}\right)
\left( q\right) _{n}} & \text{if} & m=n%
\end{array}%
\right. ,  \label{Intb^2}
\end{equation}%
where 
\begin{equation}
f_{C}(x|\beta ,q)\allowbreak =\allowbreak \frac{(\beta ^{2})_{\infty }}{%
(\beta ,\beta q)_{\infty }}f_{h}\left( x|q\right) /\prod_{j=0}^{\infty
}l\left( x|\beta q^{j}\right) ,  \label{fC}
\end{equation}%
with, as before, $l(x|a)\allowbreak =\allowbreak (1+a)^{2}-4x^{2}a$. Hence, 
\begin{equation*}
\hat{C}_{n}(\beta ,q)\allowbreak =\allowbreak \frac{\left( \beta ^{2}\right)
_{n}(1-\beta )}{\left( 1-\beta q^{n}\right) \left( q\right) _{n}}.
\end{equation*}
Summarizing we get the following result.

\begin{theorem}
i) For all $n\geq 1$ and complex $\left\vert q\right\vert <1$, $\gamma $ and 
$\beta \notin \left\{ 1,q^{-1},q^{-2},\ldots \right\} :$%
\begin{equation*}
0=\sum_{j=0}^{n}\QATOPD[ ] {n}{j}_{q}\gamma ^{j}\beta ^{n-j}\left( \beta
/\gamma \right) _{j}\left( \gamma /\beta \right) _{n-j}\frac{(1-\beta
q^{2j})\left( \gamma q^{j+1}\right) _{n-1}}{\left( \beta q^{j}\right) _{n+1}}%
.
\end{equation*}

ii) For real $\left\vert x\right\vert <1$, $\left\vert \beta \right\vert <1$%
, $\left\vert \gamma \right\vert <1$, $\left\vert q\right\vert <1:$ 
\begin{eqnarray*}
&&\frac{(\beta q)_{\infty }^{2}\left( \gamma ^{2}\right) _{\infty }}{\left(
\beta ^{2}\right) _{\infty }(\gamma q)_{\infty }^{2}}\prod_{j=0}^{\infty }%
\frac{l\left( x|\beta q^{j}\right) }{l\left( x|\gamma q^{j}\right) } \\
&=&\sum_{n\geq 0}\frac{\beta ^{n}\left( \gamma /\beta \right) _{n}\left(
\gamma \right) _{n}(1-\beta )(1-\gamma q^{2n})\left( q\right) _{2n}}{\left(
q\right) _{n}\left( \beta \right) _{n+1}(1-\gamma )\left( \gamma ^{2}\right)
_{2n}}C_{2n}(x|\gamma ,q).
\end{eqnarray*}
\end{theorem}

\begin{proof}
After applying the idea of expansion presented in the introduction, we get%
\begin{equation*}
f_{C}(x|\beta ,q)=f_{C}(x|\gamma ,q)\sum_{n\geq 0}\frac{\beta ^{n}\left(
\gamma /\beta \right) _{n}\left( \gamma \right) _{n}(1-\beta )(1-\gamma
q^{2n})\left( q\right) _{2n}}{\left( q\right) _{n}\left( \beta \right)
_{n+1}(1-\gamma )\left( \gamma ^{2}\right) _{2n}}C_{2n}(x|\gamma ,q).
\end{equation*}%
Now we cancel out $f_{h}$ on both sides and multiply both sides by%
\begin{equation*}
(1-\gamma )(\beta ,\beta q)_{\infty }\prod_{j=0}^{\infty }l\left( x|\beta
q^{j}\right) /(1-\beta )\left( \beta ^{2}\right) _{\infty },
\end{equation*}%
we get ii). To get i) we apply the idea behind (\ref{01}), use the
simplifying ratios $\left( \beta \right) _{j}/\left( \beta \right)
_{n+j+1}\allowbreak =\allowbreak 1/\left( \beta q^{j}\right) _{n+1}$, $%
\left( \gamma \right) _{n+j}/\left( \gamma \right) _{j+1}\allowbreak
=\allowbreak \left( \gamma q^{j+1}\right) _{n-1}$ and finally multiply both
sides by $\left( q\right) _{n}$. The fact that the identity has a form of
rational function, we can extend the range of unknowns $\beta $ and $\gamma $
with additional condition that expression $1-\beta q^{j}$ is not equal zero
for all $j=0,\ldots $.
\end{proof}

\subsection{$q-$Hermite and $q-$ultraspherical (Rogers)}

This is a particular case of the previous subsection. However, we consider
it separately because of the importance of the $q-$Hermite polynomials.

We have the following result.

\begin{theorem}
i) For all $n\geq 1$ and complex $q$ and $\beta \notin \left\{
1,q^{-1},\ldots \right\} :$%
\begin{eqnarray}
0 &=&\sum_{k=0}^{n}(-1)^{k}q^{\binom{k}{2}}\QATOPD[ ] {n}{k}_{q}\left( \beta
q^{n-k}\right) _{n-1},  \label{1p1} \\
0 &=&\sum_{k=0}^{n}(-1)^{k}q^{\binom{k}{2}}\QATOPD[ ] {n}{k}_{q}\frac{%
(1-\beta q^{2k})}{\left( \beta q^{k}\right) _{n+1}},  \label{1p2} \\
0 &=&\sum_{j=0}^{n}\QATOPD[ ] {n}{j}_{q}\left( \beta \right) _{j}\beta
^{n-j}\left( \beta ^{-1}\right) _{n-j}.  \label{1p3}
\end{eqnarray}

ii) For all $\left\vert \beta \right\vert ,\left\vert q\right\vert <1:$ 
\begin{eqnarray*}
\frac{(\beta ,\beta q)_{\infty }\left( -\beta \right) _{\infty }^{2}}{\left(
\beta ^{2}\right) _{\infty }} &=&\sum_{n\geq 0}(\beta )^{n}q^{\binom{n}{2}}%
\frac{\left( \beta \right) _{n}(1-\beta q^{2n})\left( q\right) _{2n}(\beta
^{2}|q^{2})_{n}}{\left( q\right) _{n}\left( \beta ^{2}\right) _{2n}(1-\beta
)(q^{2}|q^{2})_{n}}, \\
\frac{\left( \beta ^{2}\right) _{\infty }}{(\beta ,\beta q)_{\infty }\left(
-\beta \right) _{\infty }^{2}} &=&\sum_{n\geq 0}(-\beta )^{n}\frac{(1-\beta
)\left( q|q^{2}\right) _{n}}{\left( q\right) _{n}\left( \beta \right) _{n+1}}%
.
\end{eqnarray*}
\end{theorem}

\begin{proof}
Setting once $\gamma \allowbreak =\allowbreak 0$ with any $\left\vert \beta
\right\vert <1$ and then $\beta \allowbreak =\allowbreak 0$ and any $%
\left\vert \gamma \right\vert <1$ in (\ref{CnaC}) we end up with
coefficients 
\begin{equation*}
c_{0,j}=q^{\binom{k}{2}}\frac{(-\beta )^{k}\left( \beta \right) _{k}}{\left(
q\right) _{k}}~\text{and~}\hat{c}_{0,j}=\frac{\beta ^{k}\left( q\right)
_{2k}\left( 1-\beta \right) }{\left( q\right) _{k}\left( \beta \right) _{k+1}%
},
\end{equation*}
for $j\allowbreak =\allowbreak 2k$ and $0$ otherwise. From these two
expansions follow (\ref{1p1}) and (\ref{1p2}). In order to get (\ref{1p3})
let us recall also Proposition 3.1 of \cite{Szab-bAW}. Keeping in mind
formula (\ref{pC}), definition (3.5) of \cite{Szab-bAW} and its assertions
i) and iv) we conclude that for $\beta \neq 0$ and $q\neq 0$ we get 
\begin{equation}
0=\sum_{j=0}^{n}C_{j}(x|\beta ,q)\beta ^{n-j}C_{n-j}(x|\beta
^{-1},q)=\sum_{j=0}^{n}C_{j}(x|\beta ,q)q^{-n+j}C_{n-j}(x|\beta ,q^{-1}).
\label{CC}
\end{equation}%
Now, setting $x\allowbreak =\allowbreak 0$ in (\ref{CC}) we get after
multiplying both sides by $\left( q^{2}|q^{2}\right) _{n}$ and canceling out 
$(-1)^{n}$%
\begin{eqnarray*}
0\allowbreak &=&\allowbreak \sum_{j=0}^{n}\QATOPD[ ] {n}{j}_{q^{2}}\left(
\beta ^{2}|q^{2}\right) _{j}\beta ^{2(n-j)}\left( \beta ^{-2}|q^{2}\right)
_{n-j}, \\
0 &=&\sum_{j=0}^{n}\frac{\left( q^{2}|q^{2}\right) _{n}}{\left(
q^{2}|q^{2}\right) _{j}\left( q^{-2}|q^{-2}\right) _{n-j}}\left( \beta
^{2}|q^{2}\right) _{j}q^{-2n+2j}\left( \beta ^{2}|q^{-2}\right) _{n-j}.
\end{eqnarray*}%
In the first and in the second of these identities, we simply replace $q^{2}$
by $q$ and $\beta ^{2}$ by $\beta $ and also in the second we apply
well-known formula 
\begin{equation*}
\left( a|q^{-1}\right) \allowbreak =\allowbreak (-a)^{n}q^{-\binom{n}{2}%
}\left( a^{-1}\right) _{n}
\end{equation*}
getting in both cases (\ref{1p3}).

In order to get ii) we recall (\ref{Intb^2}) and get:%
\begin{eqnarray*}
f_{h}(x|q) &=&f_{C}(x|\beta ,q)\sum_{n\geq 0}(-\beta )^{n}q^{\binom{n}{2}}%
\frac{\left( \beta \right) _{n}(1-\beta q^{2n})\left( q\right) _{2n}}{\left(
q\right) _{n}\left( \beta ^{2}\right) _{2n}(1-\beta )}C_{2n}(x|\beta ,q), \\
f_{C}(x|\beta ,q) &=&f_{h}(x|q)\sum_{n\geq 0}\beta ^{n}\frac{h_{2n}(x|q)}{%
\left( q\right) _{n}\left( \beta q\right) _{n}}.
\end{eqnarray*}

Now, let us recall that 
\begin{equation*}
h_{2j}(0|q)\allowbreak =\allowbreak (-1)^{j}\left( q|q^{2}\right) _{j}
\end{equation*}%
and (following three-term recurrence satisfied by the polynomials $\left\{
C_{n}\right\} $, presented, e.g., in \cite{IA}) we have 
\begin{equation*}
C_{2n}(0|\beta ,q)\allowbreak =\allowbreak (-1)^{n}\frac{\left( \beta
^{2}|q^{2}\right) _{n}}{\left( q^{2}|q^{2}\right) _{n}}
\end{equation*}%
and finally noticing that 
\begin{equation*}
\prod_{j=0}^{\infty }l(0|\beta q^{j})\allowbreak =\allowbreak \left( -\beta
\right) _{\infty }^{2},
\end{equation*}%
we get after canceling out $f_{h}$ on both sides. This formula can be
slightly more simplified using the fact that 
\begin{equation*}
\left( \alpha \right) _{m}\left( -\alpha \right) _{m}\allowbreak
=\allowbreak \left( \alpha ^{2}|q^{2}\right) _{m}.
\end{equation*}
\end{proof}

\begin{remark}
There exist other expansions involving$\ $Rogers and $q-$Hermite
polynomials. They can be derived from the relationship between the so-called
Al-Salam-Chihara (ASC) polynomials considered for complex conjugate
parameters and the $q-$ultra\allowbreak spherical polynomials. Namely, we
have 
\begin{equation}
p_{n}(x|x,\beta ,q)=\left( q\right) _{n}C_{n}(x|\beta ,q),  \label{pC}
\end{equation}%
where $\left\{ p_{n}(x|y,\beta ,q)\right\} $ are the ASC polynomial, defined
say in \cite{IA},\cite{KLS} but considered and analyzed in more details for
complex conjugate parameters in \cite{Szab-bAW}(sec.3) satisfying three-term
recurrence given by (3.2) in \cite{Szab-bAW}. Now following formulae from
Lemma 3.1 of \cite{Szab-bAW}, we end up with the following relationships:%
\begin{eqnarray}
h_{n}(x|q)/(q)_{n} &=&\sum_{j=0}^{n}C_{j}(x|\beta ,q)\beta
^{n-j}h_{n-j}(x|q)/\left( q\right) _{n-j},  \label{hC} \\
\left( q\right) _{n}C_{n}(x|\beta ,q) &=&\sum_{j=0}^{n}\QATOPD[ ] {n}{j}%
_{q}h_{j}(x|q)\beta ^{n-j}b_{n-j}(x|q),  \label{Ch}
\end{eqnarray}%
where $\left\{ b_{j}(x|q)\right\} $ are some auxiliary polynomials related
to $q-$Hermite polynomials by formula given in \cite{Szab-bAW}(Lemma 3.1 i))
(see also (\ref{bn}), below)).

Now recall, that%
\begin{eqnarray*}
C_{2n}(0|\beta ,q)\allowbreak &=&\allowbreak (-1)^{n}\frac{\left( \beta
^{2}|q^{2}\right) _{n}}{\left( q^{2}|q^{2}\right) _{n}},~h_{2j}(0|q)%
\allowbreak =\allowbreak (-1)^{j}\left( q|q^{2}\right) _{j}, \\
\text{and }b_{2n}(0|q)\allowbreak &=&\allowbreak q^{n(n-1)}\left(
q|q^{2}\right) _{n}.
\end{eqnarray*}%
After setting these values into (\ref{hC}) and (\ref{Ch}), we get%
\begin{eqnarray*}
\frac{\left( q|q^{2}\right) _{n}}{(q)_{2n}} &=&\sum_{j=0}^{n}\frac{\left(
\beta ^{2}|q^{2}\right) _{j}}{\left( q^{2}|q^{2}\right) _{j}}\beta ^{2(n-j)}%
\frac{\left( q|q^{2}\right) _{n-j}}{(q)_{2(n-j)}}, \\
\left( q\right) _{2n}(-1)^{n}\frac{\left( \beta ^{2}|q^{2}\right) _{n}}{%
\left( q^{2}|q^{2}\right) _{n}} &=&\sum_{j=0}^{n}\QATOPD[ ] {2n}{2j}%
_{q}(-1)^{j}\left( q|q^{2}\right) _{j}\beta ^{2(n-j)}q^{(n-j)(n-j-1)}\left(
q|q^{2}\right) _{n-j}.
\end{eqnarray*}%
These identities can be easily simplified to the known ones, using the
well-known property 
\begin{equation*}
\left( q\right) _{2n}\allowbreak =\allowbreak \left( q|q^{2}\right)
_{n}\left( q^{2}|q^{2}\right) _{n}
\end{equation*}%
and replacing $\beta ^{2}$ by $\rho $ and $q^{2}$ by $q:$
\end{remark}

\begin{eqnarray*}
1 &=&\sum_{j=0}^{n}\QATOPD[ ] {n}{j}_{q}\left( \rho \right) _{j}\rho ^{n-j},
\\
(\rho )_{n} &=&\sum_{j=0}^{n}\QATOPD[ ] {n}{j}_{q}(-1)^{j}q^{\binom{j}{2}%
}\rho ^{j}.
\end{eqnarray*}

\begin{remark}
Identity (\ref{1p3}) is a particular case of the identity from Exercise
1.3(i) of \cite{GR04}, we take $a\allowbreak =\allowbreak \beta $ and $%
b\allowbreak =\allowbreak \beta .$
\end{remark}

\subsection{ $q-$ultraspherical (Rogers) and Chebyshev of the second kind}

Let us recall that $C_{n}(x|q,q)\allowbreak =\allowbreak U_{n}(x)$. Thus,
using the formula (\ref{CnaC}) we deduce that we have 
\begin{eqnarray*}
c_{j,0} &=&q^{k}\frac{\left( \beta /q\right) _{k}\left( \beta \right)
_{k}(1-q)}{\left( q\right) _{k}\left( q\right) _{k+1}}, \\
\bar{c}_{j,0} &=&\frac{\beta ^{k}\left( q/\beta \right) _{k}\left( 1-\beta
\right) }{\left( \beta \right) _{k+1}},
\end{eqnarray*}%
for $j\allowbreak =\allowbreak 2k$ and $0$ otherwise. The following result
follows from these two expansions and the formulae for the densities that
make Rogers and Chebyshev polynomials orthogonal.

\begin{theorem}
i) For all $n\geq 1$ complex $q\neq 1$ and $\beta \notin \left\{
1,q^{-1},\ldots \right\} :$%
\begin{eqnarray*}
0 &=&\sum_{k=0}^{n}\QATOPD[ ] {n}{k}_{q}\QATOPD[ ] {n+k}{k}_{q}q^{k}\beta
^{n-k}\left( \beta /q\right) _{k}\left( q/\beta \right) _{n-k}\frac{(1-\beta
q^{2k})}{\left( \beta q^{k}\right) _{n+1}(1-q^{k+1})}, \\
0 &=&\sum_{k=0}^{n}\QATOPD[ ] {2n+1}{n-k}_{q}\beta ^{k}q^{n-k}\left( q/\beta
\right) _{k}\left( \beta /q\right) _{n-k}(1-q^{2k+1})\left( \beta
q^{k+1}\right) _{n-1}.
\end{eqnarray*}

ii) For $\left\vert q\right\vert ,\left\vert \beta \right\vert <1$ 
\begin{eqnarray}
\frac{\left( q\beta ^{2}|q^{2}\right) _{\infty }(-q)_{\infty }(q^{2}|q^{2})}{%
\left( \beta ^{2}|q^{2}\right) _{\infty }} &=&\sum_{n\geq 0}(-1)^{n}\frac{%
\beta ^{n}\left( q/\beta \right) _{n}}{\left( \beta \right) _{n+1}},  \notag
\\
\frac{\left( \beta ^{2}|q^{2}\right) _{\infty }(1-\beta )^{2}}{\left( q\beta
^{2}|q^{2}\right) _{\infty }(-q)_{\infty }(q^{2}|q^{2})\left( 1-q\right) }
&=&\sum_{n\geq 0}(-1)^{n}\frac{q^{n}\left( \beta /q\right) _{n}\left( \beta
\right) _{n}\left( q|q^{2}\right) _{n}(1-\beta q^{2n})}{\left( q\right)
_{n}\left( q\right) _{n+1}\left( \beta ^{2}q|q^{2}\right) _{n}},  \notag \\
\frac{(\beta ^{2})_{\infty }(q)_{\infty }^{3}}{\left( \beta \right) _{\infty
}^{4}} &=&\sum_{n\geq 0}(2n+1)\frac{\beta ^{n}\left( q/\beta \right) _{n}}{%
\left( \beta \right) _{n+1}}.  \label{nice}
\end{eqnarray}
\end{theorem}

\begin{proof}
i) Using (\ref{CnaC}), as before, we get the following two finite expansions:%
\begin{eqnarray*}
U_{n}(x) &=&\sum_{k=0}^{\left\lfloor n/2\right\rfloor }\beta ^{k}\frac{%
\left( q/\beta \right) _{k}(q)_{n-k}(1-\beta q^{n-2k)}}{(q)_{k}(\beta
)_{n-k+1}}C_{n-2k}(x|\beta ,q), \\
C_{n}(x|\beta ,q) &=&\sum_{k=0}^{\left\lfloor n/2\right\rfloor }q^{k}\frac{%
(\beta /q)_{k}(\beta )_{n-k}(1-q^{n-2k+1})}{\left( q\right) _{k}(q)_{n-k+1}}%
U_{n-2k}(x).
\end{eqnarray*}

Hence, coefficients $c_{0,j}$ and $\hat{c}_{0,j}$ are equal to $q^{k}\frac{%
\left( \beta /q\right) _{k}\left( \beta \right) _{k}(1-q)}{\left( q\right)
_{k}\left( q\right) _{k+1}}$ and $\frac{\beta ^{k}\left( q/\beta \right)
_{k}\left( q\right) _{k}\left( 1-\beta \right) }{\left( q\right) _{k}\left(
\beta \right) _{k+1}}$ for $j\allowbreak =\allowbreak 2k$ and $0$ otherwise.
The following two identities true for all $n\geq 1$, follow from these two
expansions $n\geq 1:$ 
\begin{eqnarray*}
0 &=&\sum_{k=0}^{n}q^{k}\beta ^{n-k}\frac{\left( \beta /q\right) _{k}\left(
q/\beta \right) _{n-k}\left( \beta \right) _{k}\left( q\right)
_{n+k}(1-\beta q^{2k})}{\left( q\right) _{k}(q)_{k+1}\left( q\right)
_{n-k}\left( \beta \right) _{n+k+1}}, \\
0 &=&\sum_{k=0}^{n}\beta ^{k}q^{n-k}\frac{\left( q/\beta \right) _{k}\left(
\beta /q\right) _{n-k}\left( \beta \right) _{n+k}(1-q^{2k+1})}{\left(
q\right) _{n+k+1}\left( q\right) _{n-k}\left( \beta \right) _{k+1}}.
\end{eqnarray*}

Now, we simplify it to i).

ii) Now let us recall (\ref{Intb^2}) and the fact that $\hat{U}%
_{n}\allowbreak =\allowbreak 1$, we get 
\begin{eqnarray*}
f_{C}(x|\beta ,q) &=&\frac{2}{\pi }\sqrt{1-x^{2}}\sum_{n\geq 0}\frac{\beta
^{n}\left( q/\beta \right) _{n}\left( 1-\beta \right) }{\left( \beta \right)
_{n+1}}U_{2n}(x), \\
\frac{2}{\pi }\sqrt{1-x^{2}} &=&f_{C}(x|\beta ,q)\sum_{n\geq 0}q^{n}\frac{%
\left( \beta /q\right) _{n}\left( \beta \right) _{n}(1-q)\left( q\right)
_{2n}(1-\beta q^{2n})}{\left( q\right) _{n}\left( q\right) _{n+1}(1-\beta
)\left( \beta ^{2}\right) _{2n}}C_{2n}(x|\beta ,q).
\end{eqnarray*}%
Let us recall that 
\begin{equation*}
\frac{f_{C}(x|\beta ,q)}{\frac{2}{\pi }\sqrt{1-x^{2}}}=\frac{(\beta
^{2})_{\infty }\left( q\right) _{\infty }}{(\beta ,\beta q)_{\infty }}%
\prod_{j=0}^{\infty }\frac{l\left( x|q^{j}\right) }{l\left( x|\beta
q^{j}\right) }.
\end{equation*}

Now we cancel out $\frac{2}{\pi }\sqrt{1-x^{2}}$ on both sides and set, say $%
x=0$. We get 
\begin{equation*}
U_{2n}(0)\allowbreak =\allowbreak (-1)^{n},C_{2n}(0|\beta ,q)\allowbreak
=\allowbreak (-1)^{n}\frac{\left( \beta ^{2}|q^{2}\right) _{n}}{\left(
q^{2}|q^{2}\right) _{n}},
\end{equation*}%
\begin{eqnarray*}
f_{C}(0|\beta ,q)/(2/\pi )\allowbreak &=&\allowbreak \frac{(\beta
^{2})_{\infty }\left( q\right) _{\infty }}{(\beta ,\beta q)_{\infty }}%
\prod_{j=1}^{\infty }l(0|q^{j})/\prod_{j=0}^{\infty }l\left( 0|\beta
q^{j}\right) \allowbreak \\
&=&\allowbreak \frac{(\beta ^{2})_{\infty }(q)_{\infty }(-q)_{\infty }^{2}}{%
(\beta ,\beta q)_{\infty }\left( -\beta \right) _{\infty }^{2}}\allowbreak
=\allowbreak \frac{(1-\beta )\left( q\beta ^{2}|q^{2}\right) _{\infty
}(-q)_{\infty }(q^{2}|q^{2})}{\left( \beta ^{2}|q^{2}\right) _{\infty }}.
\end{eqnarray*}
Hence, 
\begin{eqnarray*}
\frac{\left( q\beta ^{2}|q^{2}\right) _{\infty }(-q)_{\infty }(q^{2}|q^{2})}{%
\left( \beta ^{2}|q^{2}\right) _{\infty }} &=&\sum_{n\geq 0}(-1)^{n}\frac{%
\beta ^{n}\left( q/\beta \right) _{n}}{\left( \beta \right) _{n+1}}, \\
\frac{\left( \beta ^{2}|q^{2}\right) _{\infty }(1-\beta )^{2}}{\left( q\beta
^{2}|q^{2}\right) _{\infty }(-q)_{\infty }(q^{2}|q^{2})\left( 1-q\right) }
&=&\sum_{n\geq 0}(-1)^{n}\frac{q^{n}\left( \beta /q\right) _{n}\left( \beta
\right) _{n}\left( q\right) _{2n}(1-\beta q^{2n})\left( \beta
^{2}|q^{2}\right) _{n}}{\left( q\right) _{n}\left( q\right) _{n+1}(1-\beta
)\left( \beta ^{2}\right) _{2n}\left( q^{2}|q^{2}\right) _{n}} \\
&=&\sum_{n\geq 0}(-1)^{n}\frac{q^{n}\left( \beta /q\right) _{n}\left( \beta
\right) _{n}\left( q|q^{2}\right) _{n}(1-\beta q^{2n})}{\left( q\right)
_{n}\left( q\right) _{n+1}\left( \beta ^{2}q|q^{2}\right) _{n}}.
\end{eqnarray*}

Now let us consider $x\allowbreak =\allowbreak 1$. We have $%
U_{2n}(1)\allowbreak =\allowbreak 2n+1$,%
\begin{eqnarray*}
\left. f_{C}(x|\beta ,q)/(2\sqrt{1-x^{2}}/\pi )\right\vert _{x=1}\allowbreak
&=& \\
\allowbreak \frac{(\beta ^{2})_{\infty }}{(\beta ,\beta q)_{\infty }}%
\prod_{j=1}^{\infty }l(1|q^{j})/\prod_{j=0}^{\infty }l\left( 1|\beta
q^{j}\right) \allowbreak &=&\allowbreak \frac{(\beta ^{2})_{\infty }\left(
q\right) _{\infty }(q)_{\infty }^{2}}{(\beta ,\beta q)_{\infty }\left( \beta
\right) _{\infty }^{2}}.
\end{eqnarray*}

Consequently, we get (\ref{nice}).
\end{proof}

\subsection{Al-Salam-Chihara and $q-$Hermite}

Al-Salam-Chihara polynomials (briefly ASC polynomials and traditionally
denoted by the letter $Q$) are the two-parameter family of polynomials
defined, e.g., in \cite{KLS} (subsection 14.8), in \cite{IA}(subsection
15.1). In \cite{Szab22}, however they were analyzed in great detail. In
particular, the case of two complex conjugate parameters was analyzed and
the applications of these polynomials in probability theory were pointed
out. There one can read that $\hat{Q}_{n}(a,b,q)\allowbreak =\allowbreak
\left( q,ab\right) _{n}.$

To simplify calculations, again, we will confine ourselves to consideration
of these polynomials for the parameters $a$ and $b$ that are complex
conjugate and both satisfying $\left\vert a\right\vert ,\left\vert
b\right\vert <1$. Then let us denote 
\begin{equation*}
p_{n}(x|y,\rho ,q)\allowbreak =\allowbreak Q_{n}(x|a,b,q),
\end{equation*}%
where the parameters $y$ and $\rho $ are defined by the equations $%
a+b\allowbreak =\allowbreak 2\rho y$ and $ab\allowbreak =\allowbreak \rho
^{2}$. Then polynomials $\left\{ p_{n}\right\} $ satisfy three-term
recurrence given by formula (3.2) of \cite{Szab-bAW} with initial conditions 
$p_{-1}(x|y,\rho ,q)\allowbreak =\allowbreak 0$ and $p_{0}(x|y,\rho
,q)\allowbreak =$\allowbreak $1$. In the sequel will appear a family of
auxiliary polynomials denoted $\left\{ b_{n}(x|q)\right\} _{n\geq 0}$.
Polynomials $\left\{ b_{n}\right\} $ satisfy certain three-term recurrence
given e.g. \cite{Szab-bAW} (Lemma 3.1i)) or earlier in \cite{bms}. However,
the simplest seems to be the following definition of these polynomials:%
\begin{equation}
b_{n}(x|q)\allowbreak =\allowbreak (-1)^{n}q^{\binom{n}{2}}h_{n}(x|q),
\label{bn}
\end{equation}%
for $q\neq 0$ and $b_{0}(x|0)\allowbreak =\allowbreak b_{2}(x|0)\allowbreak
=\allowbreak 1$, $b_{1}(x|0)\allowbreak =\allowbreak -2x$, and $%
b_{n}(x|0)\allowbreak =\allowbreak 0$ for $n\allowbreak =\allowbreak
-1,3,4,\ldots $ .

We have a proposition summarizing essential information on these families of
polynomials.

\begin{proposition}
For $\left\vert y\right\vert ,\left\vert \rho \right\vert <1$, we have the
following sets of CC between ASC and $q-$Hermite polynomials:%
\begin{eqnarray*}
c_{n,j}(y,\rho |q)\allowbreak &=&\allowbreak \QATOPD[ ] {n}{j}_{q}\rho
^{n-j}b_{n-j}(x|q), \\
\bar{c}_{n,j}(y,\rho |q) &=&\QATOPD[ ] {n}{j}_{q}\rho ^{n-j}h_{n-j}(y|q).
\end{eqnarray*}%
Hence, in particular $c_{n,0}(y,\rho |q)\allowbreak =\allowbreak \rho
^{n}b_{n}(y|q)$ and $\bar{c}_{n,0}(y,\rho |q)\allowbreak =\allowbreak \rho
^{n}h_{n}(y|q).$

Besides we have $\hat{p}_{n}(y,\rho ,q)\allowbreak =\allowbreak \left(
q,\rho ^{2}\right) _{n}$ and, as before, $\hat{h}_{n}(q)\allowbreak
=\allowbreak \left( q\right) _{n}.$
\end{proposition}

\begin{proof}
Following \cite{bms} and are presented, e.g., in \cite{Szab-bAW} (Lemma3.1)
we have for all $n\geq 0$ 
\begin{eqnarray}
p_{n}(x|y,\rho ,q)\allowbreak &=&\allowbreak \sum_{j=0}^{n}\QATOPD[ ] {n}{j}%
_{q}\rho ^{n-j}b_{n-j}(x|q)h_{j}(x|q),  \label{pnah} \\
h_{n}(x|q) &=&\sum_{j=0}^{n}\QATOPD[ ] {n}{j}_{q}\rho
^{n-j}h_{n-j}(y|q)p_{j}(x|y,\rho ,q).  \label{hanap}
\end{eqnarray}
\end{proof}

This case leads to the well-known and important identities. Namely, we have
the following result.

\begin{theorem}
i) For all $n\geq 1$ and complex $x,y,\rho $, we have%
\begin{equation*}
0=\sum_{j=0}^{n}\QATOPD[ ] {n}{j}_{q}h_{j}(y)b_{n-j}(y).
\end{equation*}

ii) For $\left\vert \rho \right\vert ,\left\vert q\right\vert <1$, $%
\left\vert x\right\vert ,\left\vert y\right\vert \leq 1$ 
\begin{eqnarray}
\frac{\left( \rho ^{2}\right) _{\infty }}{\prod_{j=0}^{\infty }w(x,y|\rho
q^{j})} &=&\sum_{j\geq 0}\rho ^{j}h_{j}(x|q)h_{j}(y|q)/\left( q\right) _{j},
\label{asc1} \\
\frac{\prod_{j=0}^{\infty }w(x,y|\rho q^{j})}{\left( \rho ^{2}\right)
_{\infty }} &=&\sum_{j\geq 0}\rho ^{j}b_{j}(y|q)p_{j}(x|y,\rho ,q)/\left(
q,\rho ^{2}\right) _{j}.  \label{asc2}
\end{eqnarray}
\end{theorem}

\begin{remark}
Notice that (\ref{asc1}) it is nothing else as the famous Poisson-Mehler
formula.
\end{remark}

\begin{proof}
Recall that, as shown also in \cite{Szab22}(formula 5.6), the density that
makes these polynomial orthogonal is given by the following formula%
\begin{equation}
f_{CN}(x|y,\rho ,q)\allowbreak =\allowbreak f_{h}(x|q)\frac{\left( \rho
^{2}\right) _{\infty }}{\prod_{j=0}^{\infty }w(x,y|\rho q^{j})},  \label{fCN}
\end{equation}%
where $w$ is given by \ref{wxya}. By the way the density $f_{CN}$ will be
called conditional $q-$normal since it has a clear probabilistic
interpretation as shown, e.g., in \cite{SzablAW}.

From the point of view of the main idea of this paper, the connection
coefficients of polynomials $\left\{ p_{n}\right\} $ and $\left\{
h_{n}\right\} $ are important. Hence, we have i).

ii) is obtained after applying (\ref{basicEx}) and canceling out on both
sides $f_{h}(x|q).$
\end{proof}

\begin{remark}
Now let us set $x=y$ in these identities and then note that 
\begin{equation*}
p_{n}(x|x,\rho ,q)\allowbreak =\allowbreak C_{n}(x|\rho ,q)~\text{and~}%
w(x,x|\rho )\allowbreak =\allowbreak (1-\rho )^{2}l(x|\rho ).
\end{equation*}%
We get then the following expansions of $\prod_{j=0}^{\infty }l(x|\rho
q^{j}) $ and $1/\prod_{j=0}^{\infty }l(x|\rho q^{j})$ in terms of $q-$%
Hermite and related polynomials. These expansions are true, of course, for $%
\left\vert x\right\vert ,\left\vert \rho \right\vert ,\left\vert
q\right\vert \,<1:$ 
\begin{eqnarray*}
\frac{\left( \rho ^{2}\right) _{\infty }}{\left( \rho \right) _{\infty
}^{2}\prod_{j=0}^{\infty }l(x|\rho q^{j})} &=&\sum_{j\geq 0}\rho
^{j}h_{j}^{2}(x|q)/\left( q\right) _{j}, \\
\frac{\left( \rho \right) _{\infty }^{2}\prod_{j=0}^{\infty }l(x|\rho q^{j})%
}{\left( \rho ^{2}\right) _{\infty }} &=&\sum_{j\geq 0}\rho
^{j}b_{j}(x|q)C_{j}(x|,\rho ,q)/\left( q,\rho ^{2}\right) _{j}.
\end{eqnarray*}
\end{remark}

\subsection{Askey-Wilson and Al-Salam-Chihara}

Askey-Wilson polynomials were introduced and analyzed in \cite{AW85}. Their
definition and same basic properties can be found, e.g., in \cite{KLS}%
(subsection 14.3) or \cite{Szab22}(section 6). In \cite{Szab22} the CC
between these two families of polynomials were presented in exact, legible
but not simple forms (see \cite{Szab22}(2.15 and 2.16)) that base, of
course, on the famous formula of \cite{AW85}. However, if one introduces new
parameters forming complex conjugate pairs then the expression for the CC's
can be simplified and expressed in the form of certain, well-described
polynomials. Besides, the parameters forming complex conjugate pairs have
nice probabilistic interpretation. In \cite{Szab-bAW} many simplifications
including connection coefficients were found. Thus, let us recall these
results and derive some finite and infinite identities involving them and
ASC polynomials. The new parameters are defined by the equalities%
\begin{eqnarray}
2\rho _{1}y &=&a+b,~\rho _{1}^{2}\allowbreak =\allowbreak ab,  \label{abyr}
\\
2\rho _{2}z\allowbreak &=&\allowbreak c+d,~\rho _{2}^{2}=cd.  \label{cdzr}
\end{eqnarray}%
We also denote by $\alpha _{n}(x|y,\rho _{1},z,\rho _{2},q)$ the
Askey-Wilson polynomials with new parameters

\begin{proposition}
Following (\ref{ex_w_na_p}) and (\ref{ex_p_na_w}) we have for $n\geq j\geq
0: $ 
\begin{eqnarray*}
c_{n,j}(y,\rho _{1},z,\rho _{2}|q) &=&\QATOPD[ ] {n}{j}_{q}\frac{\rho
_{2}^{n-j}\left( \rho _{1}^{2}q^{j}\right) _{n-j}}{\left( \rho _{1}^{2}\rho
_{2}^{2}q^{n+j-1}\right) _{n-j}}g_{n-j}\left( z|y,\rho _{1}\rho
_{2}q^{n-1},q\right) , \\
\bar{c}_{n,j}(y,\rho _{1},z,\rho _{2}|q) &=&\QATOPD[ ] {n}{j}_{q}\frac{\rho
_{2}^{n-j}\left( \rho _{1}^{2}q^{j}\right) _{n-j}}{\left( \rho _{1}^{2}\rho
_{2}^{2}q^{2j}\right) _{n-j}}p_{n-j}\left( z|y,\rho _{1}\rho
_{2}q^{j},q\right) .
\end{eqnarray*}%
where $g_{n}\left( z|y,\tau ,q\right) \allowbreak \ $is defined by 
\begin{equation}
g_{n}\left( x|y,\rho ,q\right) \allowbreak =\allowbreak \left\{ 
\begin{array}{ccc}
\rho ^{n}p_{n}\left( y|x,\rho ^{-1},q\right) & \text{if} & \rho \neq 0, \\ 
b_{n}\left( x|q\right) & \text{if} & \rho =0.%
\end{array}%
\right.  \label{_g}
\end{equation}%
In particular, we have%
\begin{eqnarray*}
\bar{c}_{n,0}\allowbreak &=&\allowbreak \frac{\rho _{2}^{n}\left( \rho
_{1}^{2}\right) _{n}}{\left( \rho _{1}^{2}\rho _{2}^{2}\right) _{n}}%
p_{n}(z|y,\rho _{1}\rho _{2},q), \\
c_{n,0}\allowbreak &=&\allowbreak \frac{\rho _{2}^{n}\left( \rho
_{1}^{2}\right) _{n}}{\left( \rho _{1}^{2}\rho _{2}^{2}q^{n-1}\right) _{n}}%
g_{n}(z|y,\rho _{1}\rho _{2}q^{n-1},q).
\end{eqnarray*}%
Besides, we have%
\begin{eqnarray}
\hat{p}_{n}(x|y,\rho _{1},q)\allowbreak &=&\allowbreak \left( q,\rho
_{1}^{2}\right) _{n},  \notag \\
\hat{\alpha}_{n}(x|y,\rho _{1},z,\rho _{2},q)\allowbreak &=&\frac{\left(
\rho _{1}^{2}\rho _{2}^{2}q^{n-1}\right) _{n}\left( \rho _{1}^{2},\rho
_{2}^{2},q\right) _{n}}{\left( \rho _{1}^{2}\rho _{2}^{2}\right) _{2n}}%
\prod_{j=0}^{n-1}w(z,y|\rho _{1}\rho _{2}q^{j}).  \label{ahat}
\end{eqnarray}
\end{proposition}

\begin{proof}
Following \cite{Szab-bAW}(3.10,3.11) we have 
\begin{gather}
\alpha _{n}\left( x|y,\rho _{1},z,\rho _{2},q\right) =  \label{ex_w_na_p} \\
\sum_{j=0}^{n}\left[ \QATOP{n}{j}\right] _{q}p_{j}\left( x|y,\rho
_{1},q\right) \frac{\rho _{2}^{n-j}\left( \rho _{1}^{2}q^{j}\right) _{n-j}}{%
\left( \rho _{1}^{2}\rho _{2}^{2}q^{n+j-1}\right) _{n-j}}g_{n-j}\left(
z|y,\rho _{1}\rho _{2}q^{n-1},q\right) ,  \notag
\end{gather}%
\begin{equation}
p_{n}\left( x|y,\rho _{1},q\right) =\sum_{j=0}^{n}\left[ \QATOP{n}{j}\right]
_{q}\alpha _{j}\left( x|y,\rho _{1},z,\rho _{2},q\right) \frac{\rho
_{2}^{n-j}\left( \rho _{1}^{2}q^{j}\right) _{n-j}}{\left( \rho _{1}^{2}\rho
_{2}^{2}q^{2j}\right) _{n-j}}p_{n-j}\left( z|y,\rho _{1}\rho
_{2}q^{j},q\right) .  \label{ex_p_na_w}
\end{equation}

Now, the first five assertions are obvious. The only argument is required to
justify the expression for $\hat{\alpha}_{n}$. However, it follows almost
directly from formula 7.2 of \cite{Szab22}. Now, we will use the notation 
\begin{eqnarray*}
a\allowbreak &=&\allowbreak \rho _{1}\exp (\chi ),b\allowbreak =\allowbreak
\rho _{1}\exp (-\chi ), \\
\allowbreak \allowbreak y &=&\cos \left( \chi \right) ,~c\allowbreak
=\allowbreak \rho _{2}\exp (\phi ), \\
d\allowbreak &=&\allowbreak \rho _{2}\exp (-\phi ),z=\cos \left( \phi
\right) \allowbreak .
\end{eqnarray*}%
Following (\ref{abyr}), (\ref{cdzr}) and (\ref{rozklw}), we observe that 
\begin{eqnarray*}
abcd\allowbreak &=&\allowbreak \rho _{1}^{2}\rho _{2}^{2,},ab\allowbreak
=\allowbreak \rho _{1}^{2},.cd\allowbreak =\allowbreak \rho _{2}^{2}, \\
(ac,bc,ad,bd)_{n}\allowbreak &=&\allowbreak \prod_{k-0}^{n}w(y,z|\rho
_{1}\rho _{2}q^{k}).
\end{eqnarray*}
\end{proof}

Recall also that the density that makes AW polynomials with complex
conjugate parameters is denoted $f_{C2N}$, because of its clear,
probabilistic interpretation as a certain conditional density (for details
see \cite{Szab-bAW}, \cite{Szab22} or \cite{Szab13}). In particular, it was
shown in \cite{SzablAW} that the density that makes polynomials $\left\{
\alpha _{n}\right\} $ orthogonal is 
\begin{equation}
f_{C2N}(x|y,\rho _{1},z,\rho _{2},q)\allowbreak =\allowbreak \frac{%
f_{CN}(y|x,\rho _{1},q)f_{CN}(x|z,\rho _{2},q)}{f_{CN}(y|z,\rho _{1}\rho
_{2},q)},  \label{fc2n1}
\end{equation}%
where $f_{CN}(x|y,\rho ,q)$ is given by (\ref{fCN}). Let us remark that by
the symmetry argument we have also%
\begin{equation}
f_{C2N}(x|y,\rho _{1},z,\rho _{2},q)=\frac{f_{CN}(x|y,\rho
_{1},q)f_{CN}(z|x,\rho _{2},q)}{f_{CN}(z|y,\rho _{1}\rho _{2},q)}.
\label{fc2n2}
\end{equation}%
By the way, it was also shown in \cite{SzablAW} that 
\begin{equation*}
\int_{-1}^{1}f_{CN}(x|y,\rho _{1},q)f_{CN}(y|z,\rho
_{2},q)dy=f_{CN}(x|z,\rho _{1}\rho _{2},q),
\end{equation*}%
i.e., the Chapman-Kolmogorov property holds.

Having recalled (\ref{ex_w_na_p}) and (\ref{ex_p_na_w}) we put all the
necessary information to apply ideas from the introduction into the
following summary. Hence, we have the following result:

\begin{theorem}
i) For all $n\geq 1$, complex $z$, $y$, $\beta $ such that $\beta ^{2}\notin
\{1,q^{-1},\ldots ,\}:$%
\begin{eqnarray*}
0 &=&\sum_{j=0}^{n}\QATOPD[ ] {n}{j}_{q}\left( \beta ^{2}q^{j}\right)
_{n-1}p_{j}(z|y,\beta ,q)g_{n-j}(z|y,\beta q^{n-1},q), \\
0 &=&\sum_{j=0}^{n}\QATOPD[ ] {n}{j}_{q}\frac{p_{n-j}(z|y,\beta
q^{j},q)g_{j}(z|y,\beta q^{j-1},q)(1-\beta ^{2}q^{2j-1})}{\left( \beta
^{2}q^{j-1}\right) _{n}(1-\beta ^{2}q^{n+j-1})}.
\end{eqnarray*}

In particular, for $x\allowbreak =\allowbreak y\allowbreak =\allowbreak 0$
we get (\ref{1p1}) and (\ref{1p2}), that were obtained by other means.

ii) For $\left\vert y\right\vert ,\left\vert z\right\vert ,\left\vert \rho
_{1}\right\vert ,\left\vert \rho _{2}\right\vert ,\left\vert q\right\vert <1$
we have 
\begin{gather*}
f_{CN}(x|z,\rho _{2},q)=f_{CN}(z|y,\rho _{1}\rho _{2},q)\sum_{n\geq 0}\frac{%
\rho _{2}^{n}}{\left( \rho _{1}^{2}\rho _{2}^{2}\right) _{n}\left( q\right)
_{n}}p_{n}(z|y,\rho _{1}\rho _{2},q)p_{n}(x|y,\rho _{1},q), \\
f_{CN}(z|y,\rho _{1}\rho _{2},q)=f_{CN}(x|z,\rho _{2},q)\times \\
\sum_{n\geq 0}\frac{\rho _{2}^{n}\left( \rho _{1}^{2}\rho _{2}^{2}\right)
_{2n}\prod_{j=0}^{n-1}w\left( y,z|\rho _{1}\rho _{2}q^{j}\right) }{\left(
\rho _{1}^{2}\rho _{2}^{2}q^{n-1}\right) _{n}^{2}\left( \rho
_{2}^{2},q\right) _{n}}g_{n}(z|y,\rho _{1}\rho _{2}q^{n-1},q)\alpha
_{n}(x|y,\rho _{1},z,\rho _{2},q).
\end{gather*}
\end{theorem}

\begin{proof}
i) The identities $\sum_{j=0}^{n}c_{n.j}\bar{c}_{j,0}\allowbreak
=\allowbreak 0$ and $\sum_{j=0}^{n}\bar{c}_{n,j}c_{j,0}\allowbreak
=\allowbreak 0$ imply that :

\begin{eqnarray*}
0 &=&\sum_{j=0}^{n}\QATOPD[ ] {n}{j}_{q}\frac{\rho _{2}^{n-j}\left( \rho
_{1}^{2}q^{j}\right) _{n-j}}{\left( \rho _{1}^{2}\rho
_{2}^{2}q^{n+j-1}\right) _{n-j}}g_{n-j}\left( z|y,\rho _{1}\rho
_{2}q^{n-1},q\right) \frac{\rho _{2}^{j}\left( \rho _{1}^{2}\right) _{j}}{%
\left( \rho _{1}^{2}\rho _{2}^{2}\right) _{j}}p_{j}(z|y,\rho _{1}\rho _{2},q)
\\
&=&\rho _{2}^{n}\left( \rho _{1}^{2}\right) _{n}\sum_{j=0}^{n}\QATOPD[ ] {n}{%
j}_{q}\frac{g_{n-j}\left( z|y,\rho _{1}\rho _{2}q^{n-1},q\right)
p_{j}(z|y,\rho _{1}\rho _{2},q)}{\left( \rho _{1}^{2}\rho
_{2}^{2}q^{n+j-1}\right) _{n-j}\left( \rho _{1}^{2}\rho _{2}^{2}\right) _{j}}%
, \\
0 &=&\sum_{j=0}^{n}\left[ \QATOP{n}{j}\right] _{q}\frac{\rho
_{2}^{n-j}\left( \rho _{1}^{2}q^{j}\right) _{n-j}}{\left( \rho _{1}^{2}\rho
_{2}^{2}q^{2j}\right) _{n-j}}p_{n-j}\left( z|y,\rho _{1}\rho
_{2}q^{j},q\right) \\
&&\times \frac{\rho _{2}^{j}\left( \rho _{1}^{2}\right) _{j}}{\left( \rho
_{1}^{2}\rho _{2}^{2}q^{j-1}\right) _{j}}g_{j}(z|y,\rho _{1}\rho
_{2}q^{j-1},q) \\
&=&\rho _{2}^{n}\left( \rho _{1}^{2}\right) _{n}\sum_{j=0}^{n}\left[ \QATOP{n%
}{j}\right] _{q}\frac{p_{n-j}\left( z|y,\rho _{1}\rho _{2}q^{j},q\right)
g_{j}(z|y,\rho _{1}\rho _{2}q^{j-1},q)}{\left( \rho _{1}^{2}\rho
_{2}^{2}q^{2j}\right) _{n-j}\left( \rho _{1}^{2}\rho _{2}^{2}q^{j-1}\right)
_{j}}.
\end{eqnarray*}

Now, it is enough to denote $\beta \allowbreak =\allowbreak \rho _{1}\rho
_{2}$ and apply (\ref{knk1}) and (\ref{knk2}).

ii) We have by (\ref{basicEx}) 
\begin{gather*}
f_{C2N}(x|y,\rho _{1},z,\rho _{2},q) \\
=f_{CN}(x|y,\rho _{1},q)\sum_{n\geq 0}\frac{\rho _{2}^{n}}{\left( \rho
_{1}^{2}\rho _{2}^{2}\right) _{n}\left( q\right) _{n}}p_{n}(z|y,\rho
_{1}\rho _{2},q)p_{n}(x|y,\rho _{1},q), \\
f_{CN}(x|y,\rho _{1},q)=f_{C2N}(x|y,\rho _{1},z,\rho _{2},q) \\
\times \sum_{n\geq 0}\frac{\rho _{2}^{n}\left( \rho _{1}^{2}\rho
_{2}^{2}\right) _{2n}\prod_{j=0}^{n-1}w\left( y,z|\rho _{1}\rho
_{2}q^{j}\right) }{\left( \rho _{1}^{2}\rho _{2}^{2}q^{n-1}\right)
_{n}^{2}\left( \rho _{2}^{2},q\right) _{n}} \\
\times g_{n}(z|y,\rho _{1}\rho _{2}q^{n-1},q)\alpha _{n}(x|y,\rho
_{1},z,\rho _{2},q).
\end{gather*}%
Taking into account (\ref{fc2n1}) and (\ref{fc2n2}) and cancelling out $%
f_{CN}(x|y,\rho _{1},q)$ we get%
\begin{gather*}
f_{CN}(x|z,\rho _{2},q)=f_{CN}(z|y,\rho _{1}\rho _{2},q)\sum_{n\geq 0}\frac{%
\rho _{2}^{n}}{\left( \rho _{1}^{2}\rho _{2}^{2}\right) _{n}\left( q\right)
_{n}}p_{n}(z|y,\rho _{1}\rho _{2},q)p_{n}(x|y,\rho _{1},q), \\
f_{CN}(z|y,\rho _{1}\rho _{2},q)=f_{CN}(x|z,\rho _{2},q)\times \\
\sum_{n\geq 0}\frac{\rho _{2}^{n}\left( \rho _{1}^{2}\rho _{2}^{2}\right)
_{2n}\prod_{j=0}^{n-1}w\left( y,z|\rho _{1}\rho _{2}q^{j}\right) }{\left(
\rho _{1}^{2}\rho _{2}^{2}q^{n-1}\right) _{n}^{2}\left( \rho
_{2}^{2},q\right) _{n}}g_{n}(z|y,\rho _{1}\rho _{2}q^{n-1},q)\alpha
_{n}(x|y,\rho _{1},z,\rho _{2},q).
\end{gather*}
\end{proof}

\begin{remark}
Let us notice, that these two identities are the particular cases of the two
similar identities proved in \cite{Szab-bAW}(Corollary 3.3).
\end{remark}

\begin{remark}
These identities look very complicated. One can simplify them a bit by
setting $z\allowbreak =\allowbreak y$ and remembering that $p_{n}(z|z,\rho
,q)\allowbreak =\allowbreak (q)_{n}C_{n}(x|\rho ,q)$.

Hence, we have for all $n\geq 1$ and complex $x$ and complex $\rho $ such
that $\ \rho ^{2}\notin \{q^{-j}|j=0,1,..\}:$ 
\begin{eqnarray}
0 &=&\sum_{j=0}^{n}(\rho q^{n-1})^{n-j}\left( \rho ^{2}q^{j}\right)
_{n-1}C_{j}(x|\rho ,q)C_{n-j}(x|\rho ^{-1}q^{-(n-1)},q),  \label{C1} \\
0 &=&\sum_{j=0}^{n}\frac{\left( \rho q^{j-1}\right) ^{j}C_{j}(x|\rho
^{-1}q^{-(j-1)},q)}{\left( \rho ^{2}q^{j-1}\right) _{j}}\frac{C_{n-j}(x|\rho
q^{j},q)}{\left( \rho ^{2}q^{2j}\right) _{n-j}}.  \label{C2}
\end{eqnarray}%
In order to get (\ref{C1}), we used (\ref{knk2}).

Now recall that 
\begin{equation*}
C_{2n}(0|\beta ,q)\allowbreak =\allowbreak (-1)^{n}\frac{\left( \beta
^{2}|q^{2}\right) _{n}}{\left( q^{2}|q^{2}\right) _{n}}
\end{equation*}%
and that we have consequently 
\begin{equation*}
\beta ^{2n}C_{2n}(0|\beta ^{-1},q)\allowbreak =\allowbreak
(-1)^{n}\prod_{j=0}^{n-1}(\beta ^{2}-q^{2j}).
\end{equation*}%
As a result of these observations, we get 
\begin{eqnarray*}
(\rho q^{2n-1})^{2n-2j}C_{2n-2j}(0|\rho ^{-1}q^{-(2n-1)},q) &=&\frac{%
q^{(n-j)(n-j-1)}\left( \rho ^{2}q^{2n+2j}|q^{2}\right) _{n-j}}{\left(
q^{2}|q^{2}\right) _{n-j}}, \\
\left( \rho q^{2j-1}\right) ^{2j}C_{2j}(0|\rho ^{-1}q^{-(2j-1)},q) &=&\frac{%
q^{j(j-1)}\left( \rho ^{2}q^{2j}|q^{2}\right) }{\left( q^{2}|q^{2}\right)
_{j}}.
\end{eqnarray*}%
Hence, we have after applying the formulae 
\begin{equation*}
\left( a\right) _{2n}\allowbreak =\allowbreak \left( a|q^{2}\right)
_{n}\left( aq|q^{2}\right) _{n}~\text{and~}\left( a\right) _{n+m}\allowbreak
=\allowbreak \left( a\right) _{n}\left( aq^{n}\right) _{m}
\end{equation*}%
and multiplying both sides by $\left( q^{2}|q^{2}\right) _{n}$. 
\begin{eqnarray*}
0 &=&\sum_{k=0}^{n}\QATOPD[ ] {n}{k}_{q^{2}}\frac{(-1)^{k}q^{(n-k)(n-k-1)}(%
\rho ^{2}|q^{2})_{k}(\rho ^{2}q^{2n+2k}|q^{2})_{n-k}}{\left( \rho
^{2}|q^{2}\right) _{k}(\rho ^{2}q|q^{2})_{k}(\rho
^{2}q^{2n+2k-1}|q^{2})_{n-k}\left( \rho ^{2}q^{2n+2k}|q^{2}\right) _{n-k}} \\
&=&\sum_{k=0}^{n}\QATOPD[ ] {n}{k}_{q^{2}}\frac{(-1)^{k}q^{2\binom{n-k}{2}}}{%
(\rho ^{2}q|q^{2})_{k}\left( \rho ^{2}q^{2n+2k-1}|q^{2}\right) _{n-k}},
\end{eqnarray*}%
and 
\begin{eqnarray*}
0 &=&\sum_{k=0}^{n}\QATOPD[ ] {n}{k}_{q^{2}}\frac{(-1)^{n-k}q^{k(k-1)}\left(
\rho ^{2}q^{2k}|q^{2}\right) _{k}\left( \rho ^{2}q^{4k}|q^{2}\right) _{n-k}}{%
\left( \rho ^{2}q^{2k-1}|q^{2}\right) _{k}\left( \rho
^{2}q^{2k}|q^{2}\right) _{k}\left( \rho ^{2}q^{4k}|q^{4k}\right)
_{n-k}\left( \rho ^{2}q^{4k+1}|q^{2}\right) _{n-k}} \\
&=&\sum_{k=0}^{n}\QATOPD[ ] {n}{k}_{q^{2}}\frac{(-1)^{n-k}q^{k(k-1)}}{\left(
\rho ^{2}q^{2k-1}|q^{2}\right) _{k}\left( \rho ^{2}q^{4k+1}|q^{2}\right)
_{n-k}}.
\end{eqnarray*}

Let us now change $q^{2}$ to $q$ and denote by $a\allowbreak =\allowbreak
\rho ^{2}q$. We get now: 
\begin{eqnarray}
0 &=&\left( a\right) _{2n-1}\sum_{k=0}^{n}\QATOPD[ ] {n}{k}_{q}\frac{%
(-1)^{k}q^{\binom{n-k}{2}}}{(a)_{k}\left( aq^{n+k-1}\right) _{n-k}}
\label{nr1} \\
&=&\sum_{k=0}^{n}\QATOPD[ ] {n}{k}_{q}(-1)^{k}q^{\binom{n-k}{2}}\left(
aq^{k}\right) _{n-1},  \notag \\
0 &=&\sum_{k=0}^{n}\QATOPD[ ] {n}{k}_{q}\frac{(-1)^{k}q^{\binom{k}{2}}}{%
\left( aq^{k-1}\right) _{k}\left( aq^{2k}\right) _{n-k}}  \label{nr2} \\
&=&\sum_{k=0}^{n}\QATOPD[ ] {n}{k}_{q}\frac{(-1)^{k}q^{\binom{k}{2}%
}(1-aq^{2k-1})}{\left( aq^{k-1}\right) _{n+1}}.  \notag
\end{eqnarray}
The last equalities hold since we have applied (\ref{knk2}) and (\ref{knk1}%
). Notice also, that the first of these identities is identical with (\ref%
{1p1}), and the second is identical with (\ref{1p2}) after setting $\beta
\allowbreak =\allowbreak a/q.$
\end{remark}

\subsection{Askey-Wilson and continuous $q-$Hermite}

Formulae (\ref{pnah}) and (\ref{hanap}) together with (\ref{ex_w_na_p}) and (%
\ref{ex_p_na_w}) allow to expand $n-$th AW polynomial (considered with
complex conjugate parameters) in the series of $q-$Hermite polynomials and
conversely. Namely, after relatively not complicated algebra we have%
\begin{eqnarray*}
\alpha _{n}\left( x|y,\rho _{1},z,\rho _{2},q\right)
&=&\sum_{k=0}^{n}h_{k}(x|q)c_{n,k}(y,\rho _{1},z,\rho _{2},q), \\
h_{n}(x|q) &=&\sum_{k=0}^{n}\alpha _{k}\left( x|y,\rho _{1},z,\rho
_{2},q\right) \bar{c}_{n,k}(y,\rho _{1},z,\rho _{2},q),
\end{eqnarray*}%
where 
\begin{gather*}
c_{n,k}(y,\rho _{1},z,\rho _{2},q)=\QATOPD[ ] {n}{k}_{q}\frac{\left( \rho
_{1}^{2}q^{k}\right) _{n-k}}{\left( \rho _{1}^{2}\rho
_{2}^{2}q^{n+k-1}\right) _{n-k}} \\
\times \sum_{s=0}^{n-k}\QATOPD[ ] {n-k}{s}_{q}\frac{\rho _{2}^{n-k-s}\rho
_{1}^{s}\left( \rho _{1}^{2}\rho _{2}^{2}q^{n+k-1}\right) _{s}}{\left( \rho
_{1}^{2}q^{k}\right) _{s}}b_{s}\left( y|q\right) g_{n-k-s}\left( z|y,\rho
_{1}\rho _{2}q^{n-1},q\right) , \\
\bar{c}_{n,k}(y,\rho _{1},z,\rho _{2},q)=\QATOPD[ ] {n}{k}_{q} \\
\times \sum_{s=0}^{n-k}\QATOPD[ ] {n-k}{s}_{q}\frac{\rho _{1}^{n-k-s}\rho
_{2}^{s}\left( \rho _{1}^{2}q^{k}\right) _{s}}{\left( \rho _{1}^{2}\rho
_{2}^{2}q^{2k}\right) _{s}}h_{n-k-s}(y|q)p_{s}(z|y,\rho _{1}\rho
_{2}q^{k},q).
\end{gather*}
Either following directly Corollary 3.1 of \cite{Szab-bAW} or applying (\ref%
{fc2n1}) and (\ref{fCN}), we get 
\begin{equation*}
f_{C2N}\left( x|y,\rho _{1},z,\rho _{2},q\right) =f_{h}\left( x|q\right) 
\frac{\left( \rho _{1}^{2},\rho _{2}^{2}\right) _{\infty
}\prod_{j=0}^{\infty }\omega \left( y,z|\rho _{1}\rho _{2}q^{j}\right) }{%
\left( \rho _{1}^{2}\rho _{2}^{2}\right) _{\infty }\prod_{j=0}^{\infty
}\omega \left( x,y|\rho _{1}q^{j}\right) \omega \left( x,z|\rho
_{2}q^{j}\right) },
\end{equation*}%
where, as before $f_{h}$ denotes the density that makes $q-$Hermite
polynomials orthogonal $q$-Hermite. Recall that $\hat{\alpha}_{n}$ is given
by \ref{ahat}.

Taking all these facts into account, we can formulate the following result:

\begin{theorem}
For complex $\left\vert y\right\vert <1$, $\left\vert z\right\vert <1$, $%
\left\vert \rho _{1}\right\vert <1$, $\left\vert \rho _{2}\right\vert <1$, $%
\left\vert q\right\vert <1$, we have

$:$i) for all $n\geq 1$, 
\begin{eqnarray*}
\sum_{k=0}^{n}c_{n,k}(y,\rho _{1},z,\rho _{2},q)\bar{c}_{j,0}(y,\rho
_{1},z,\rho _{2},q) &=&0, \\
\sum_{k=0}^{n}\bar{c}_{n,k}(y,\rho _{1},z,\rho _{2},q)c_{j,0}(y,\rho
_{1},z,\rho _{2},q) &=&0,
\end{eqnarray*}

ii) for $\left\vert x\right\vert <1$ 
\begin{gather*}
f_{C2N}(x|y,\rho _{1},z,\rho _{2},q)=f_{h}(x|q)\sum_{n\geq 0}\frac{h_{n}(x|q)%
}{\left( q\right) _{n}} \\
\times \sum_{s=0}^{n}\QATOPD[ ] {n}{s}_{q}\frac{\rho _{1}^{n-s}\rho
_{2}^{s}\left( \rho _{1}^{2}\right) _{s}}{\left( \rho _{1}^{2}\rho
_{2}^{2}\right) _{s}}h_{n-s}(y|q)p_{s}(z|y,\rho _{1}\rho _{2},q), \\
f_{h}(x|q)=f_{C2N}(x|y,\rho _{1},z,\rho _{2},q)\sum_{n\geq 0}\frac{\alpha
_{n}\left( x|y,\rho _{1},z,\rho _{2},q\right) }{\left( q\right)
_{n}\prod_{j=0}^{n-1}w(z,y|\rho _{1}\rho _{2}q^{j})}\frac{\left( \rho
_{1}^{2}\rho _{2}^{2}\right) _{2n}}{\left( \rho _{2}^{2}\right) _{n}\left(
\rho _{1}^{2}\rho _{2}^{2}q^{n-1}\right) _{n}^{2}} \\
\times \sum_{s=0}^{n}\QATOPD[ ] {n}{s}_{q}\frac{\rho _{2}^{n-s}\rho
_{1}^{s}\left( \rho _{1}^{2}\rho _{2}^{2}q^{n-1}\right) _{s}}{\left( \rho
_{1}^{2}\right) _{s}}b_{s}\left( y|q\right) g_{n-s}\left( z|y,\rho _{1}\rho
_{2}q^{n-1},q\right) .
\end{gather*}
\end{theorem}

The identities presented in this Theorem are depending on $5$ parameters
(including, $q$) and thus can be the source of many other interesting
identities if one assumes particular values of some these parameters and
leaves the others as unknowns, as it was done in previous subsections.

\subsection{Askey-Wilson and continuous dual $q-$Hahn}

The continuous dual $q-$Hahn (briefly CDqH) polynomials are described ,
e.g., in \cite{KLS}. Also, following this book we know that we pass from AW
to CDqH polynomials by setting the value of one of $4$ parameters to $0$.
These polynomials, their properties for the absolute values of the remaining 
$3$ parameters being less than $1$ were analyzed in \cite{Szab22} and \cite%
{Szab-bAW}. There are also simple, friendly connection coefficients between
these two sets of polynomials. So, let us recall for the sake of
completeness of the paper the basic definitions and the properties of these
two families of polynomials. The version of AW polynomials that $\left\{
w_{n}(x|a,b,c,d,q)\right\} $ we are going to analyze here, is given by the
three-term recurrence given in \cite{Szab22}(7.1). Formula (7.2) of \cite%
{Szab22} gives the value $\hat{\alpha}_{n}$ for this family. Namely, we have%
\begin{equation}
\hat{w}_{n}(a,b,c,d|q)=\frac{\left( abcdq^{n-1}\right) _{n}\left(
ab,ac,ad,bc,bd,cd,q\right) _{n}}{\left( abcd\right) _{2n}}.  \label{alfn}
\end{equation}%
Now recall the CDqH polynomials, denoted in the paper by $\psi $ satisfy:%
\begin{equation*}
\psi _{n}\left( x|b,c,d,q\right) \allowbreak =\allowbreak w_{n}(x|0,b,c,d,q).
\end{equation*}

Hence, we have 
\begin{equation}
\hat{\psi}_{n}\left( b,c,d|q\right) =\left( bc,bd,cd,q\right) _{n}.
\label{cdhn}
\end{equation}

We will need also the formulae for the densities that make these families of
polynomials orthogonal. Namely, following formula (7.3) of \cite{Szab22} we
have%
\begin{gather}
f_{AW}\left( x|a,b,c,d,q\right) =f_{h}\left( x|q\right) \varphi _{h}\left(
x|a,q\right) \varphi _{h}\left( x|b,q\right) \varphi _{h}\left( x|c,q\right)
\varphi _{h}\left( x|d,q\right) \times  \label{fAW} \\
\frac{\left( ab,ac,ad,bc,bd,cd\right) _{\infty }}{\left( abcd\right)
_{\infty }}.  \notag
\end{gather}%
where $f_{h}$ is given by (\ref{fh}) and 
\begin{equation*}
\varphi _{h}(x|t,q)=\frac{1}{\prod_{k=0}^{\infty }v\left( x|tq^{k}\right) }.
\end{equation*}%
with $v(x|a)\allowbreak =\allowbreak 1-2ax+a^{2}$. Hence, consequently we
have 
\begin{equation}
f_{CH}\left( x|b,c,d,q\right) =\left( bc,bd,cd\right) _{\infty }f_{h}\left(
x|q\right) \varphi _{h}\left( x|b,q\right) \varphi _{h}\left( x|c,q\right)
\varphi _{h}\left( x|d,q\right) .  \notag
\end{equation}

In \cite{Szab-bAW} (Lemma2.1) the connection coefficients between AW and
CDqH families of polynomials were given. Namely, we have:

\begin{gather}
w_{n}(x|a,b,c,d,q)\allowbreak =\allowbreak \sum_{i=0}^{n}\QATOPD[ ] {n}{i}%
_{q}\left( -a\right) ^{n-i}q^{\binom{n-i}{2}}\frac{\left(
bcq^{i},bdq^{i},cdq^{i}\right) _{n-i}}{\left( abcdq^{n+i-1}\right) _{n-i}}%
\psi _{i}\left( x|b,c,d,q\right) ,  \label{a_na_c} \\
\psi _{n}\left( x|b,c,d,q\right) =\sum_{i=0}^{n}\QATOPD[ ] {n}{i}_{q}a^{n-i}%
\frac{\left( bcq^{i},bdq^{i},cdq^{i}\right) _{n-i}}{\left( abcdq^{2i}\right)
_{n-i}}w_{i}\left( x|a,b,c,d,q\right) .  \label{c_na_a}
\end{gather}

Given this we have the following result.

\begin{theorem}
\label{AWCDH}

For all complex $\left\vert x\right\vert ,\left\vert a\right\vert
,\left\vert b\right\vert ,\left\vert c\right\vert ,\left\vert d\right\vert
,\left\vert q\right\vert <1$ we have%
\begin{gather*}
\frac{f_{AW}\left( x|a,b,c,d,q\right) }{f_{CH}\left( x|b,c,d,q\right) }=%
\frac{\left( ab,ac,ad\right) _{\infty }}{\left( abcd\right) _{\infty }}%
\varphi _{h}\left( x|a,q\right) \\
=\sum_{n\geq 0}\frac{a^{n}}{\left( abcd,q\right) _{n}}\psi _{n}\left(
x|b,c,d,q\right) , \\
\frac{f_{CH}\left( x|b,c,d,q\right) }{f_{AW}\left( x|a,b,c,d,q\right) }=%
\frac{\left( abcd\right) _{\infty }}{\left( ab,ac,ad\right) _{\infty
}\varphi _{h}\left( x|a,q\right) } \\
=\sum_{n\geq 0}\left( -a\right) ^{n}q^{\binom{n}{2}}\frac{\left( abcd\right)
_{2n}}{\left( abcdq^{n-1}\right) _{n}^{2}\left( ab,ac,ad,q\right) _{n}}%
w_{n}(x|a,b,c,d,q).
\end{gather*}
\end{theorem}

\begin{proof}
We apply expansion (\ref{basicEx}), with 
\begin{equation*}
\hat{c}_{n,0}(a,b,c,d|q)=a^{n}\frac{\left( ac,bd,cd\right) _{n}}{\left(
abcbd\right) _{n}}
\end{equation*}%
and (\ref{cdhn}) to get the first expansion and 
\begin{equation*}
c_{n,0}(a,b,c,d|q)=\left( -a\right) ^{n}q^{\binom{n}{0}}\frac{\left(
ac,bd,cd\right) _{n}}{\left( abcbdq^{n}\right) _{n}}
\end{equation*}%
and (\ref{alfn}) to get the second one.
\end{proof}

\begin{remark}
Let us notice that we can get CC $\left\{ c_{n,j}\right\} _{n\geq 1,0\leq
j\leq n}$ and $\left\{ \bar{c}_{n,j}\right\} _{n\geq 1,0\leq j\leq n}$ from (%
\ref{a_na_c}) and (\ref{c_na_a}) then apply (\ref{01}) and (\ref{02}) in
order to get the following identities%
\begin{gather}
\delta _{n,0}=\sum_{k=0}^{n}\QATOPD[ ] {n}{k}_{q}\left( -a\right) ^{n-k}q^{%
\binom{n-k}{2}}\frac{\left( bcq^{k}\right) _{n-k}\left( bdq^{k}\right)
_{n-k}\left( cdq^{k}\right) _{n-k}}{\left( abcdq^{n+k-1}\right) _{n-k}}
\label{first} \\
\times a^{k}\frac{\left( bc\right) _{k}\left( bd\right) _{k}\left( cd\right)
_{k}}{\left( abcd\right) _{k}}  \notag \\
=a^{n}\left( bc\right) _{n}\left( bd\right) _{n}\left( cd\right)
_{n}\sum_{k=0}^{n}\QATOPD[ ] {n}{k}_{q}\left( -1\right) ^{n-k}q^{\binom{n-k}{%
2}}\frac{1}{\left( abcdq^{n+k-1}\right) _{n-k}\left( abcd\right) _{k}} 
\notag
\end{gather}%
and 
\begin{gather}
\delta _{n,0}=\sum_{j=0}^{n}\QATOPD[ ] {n}{j}_{q}a^{n-j}\frac{\left(
bcq^{j}\right) _{n-j}\left( bdq^{j}\right) _{n-j}\left( cdq^{j}\right) _{n-j}%
}{\left( abcdq^{2j}\right) _{n-j}}\left( -a\right) ^{j}q^{\binom{j}{2}}\frac{%
\left( bc\right) _{j}\left( bd\right) _{j}\left( cd\right) _{j}}{\left(
abcdq^{j-1}\right) _{j}}  \label{second} \\
=a^{n}\left( bc\right) _{n}\left( bd\right) _{n}\left( cd\right)
_{n}\sum_{j=0}^{n}\QATOPD[ ] {n}{j}_{q}\left( -1\right) ^{j}\frac{q^{\binom{j%
}{2}}}{\left( abcdq^{2j}\right) _{n-j}\left( abcdq^{j-1}\right) _{j}}. 
\notag
\end{gather}%
These identities can be further simplified by dividing both sides by $%
a^{n}\left( bc\right) _{n}\left( bd\right) _{n}\left( cd\right) _{n}$. Now
let us change $abcd$ to $x$. Now apply (\ref{knk2}) and (\ref{knk1}) and get
identities (\ref{nr1}) and (\ref{nr2}). Thus, this case provides the
another, simpler justification of (\ref{nr1}) and (\ref{nr2}).
\end{remark}

\section{Proofs\label{Dow}}

\begin{proof}[Proof of Proposition \protect\ref{ccon}.]
The first four statements, i.e. ((\ref{ebb}), (\ref{oebb}), (\ref{ea1/2}), (%
\ref{oea1/2}), (\ref{ea3/2}), (\ref{oea3/2})) are obvious. Hence, let's
concentrate on the next four. We have%
\begin{gather*}
c_{n,j}(a,1/2;b,1/2)=\frac{\left( 1/2\right) ^{\left( n\right) }\left(
a-1/2+n\right) ^{\left( j\right) }}{\left( n-j\right) !\left( 1/2\right)
^{\left( j\right) }\left( b-1/2+j\right) ^{\left( j\right) }} \\
\times \sum_{k=j}^{n}\left( -1\right) ^{k-j}\binom{n-j}{k-j}\frac{\left(
a-1/2+n+j\right) ^{\left( k-j\right) }\left( 1/2+j\right) ^{\left(
k-j\right) }}{\left( 1/2+j\right) ^{\left( k-j\right) }\left(
b+1/2+2j\right) ^{\left( k-j\right) }}
\end{gather*}

\begin{gather*}
=\frac{\left( 1/2\right) ^{\left( n\right) }\left( a-1/2+n\right) ^{\left(
j\right) }}{\left( n-j\right) !\left( 1/2\right) ^{\left( j\right) }\left(
b-1/2+j\right) ^{\left( j\right) }\left( b+1/2+2j\right) ^{\left( n-j\right)
}} \\
\times \sum_{k=j}^{n}\left( -1\right) ^{k-j}\binom{n-j}{k-j}\frac{\left(
a-1/2+n+j\right) ^{\left( k-j\right) }\left( b+1/2+2j\right) ^{\left(
n-j\right) }}{\left( b+1/2+2j\right) ^{\left( k-j\right) }}
\end{gather*}

\begin{gather*}
=\frac{\left( 1/2\right) ^{\left( n\right) }\left( a-1/2+n\right) ^{\left(
j\right) }}{\left( n-j\right) !\left( 1/2\right) ^{\left( j\right) }\left(
b-1/2+j\right) ^{\left( j\right) }\left( b+1/2+2j\right) ^{\left( n-j\right)
}} \\
\times \sum_{s=0}^{n-j}\left( -1\right) ^{s}\binom{n-j}{s}\frac{\left(
a-1/2+n+j\right) ^{\left( k-j\right) }\left( b+1/2+2j\right) ^{\left(
n-j\right) }}{\left( b+1/2+2j\right) ^{\left( s\right) }} \\
=(-1)^{n-j}\frac{\left( 1/2\right) ^{\left( n\right) }\left( a-1/2+n\right)
^{\left( j\right) }}{\left( n-j\right) !\left( 1/2\right) ^{\left( j\right)
}\left( b-1/2+j\right) ^{\left( j\right) }\left( b+1/2+2j\right) ^{\left(
n-j\right) }} \\
\times \sum_{s=0}^{n-j}\left( -1\right) ^{n-j-s}\binom{n-j}{s}\left(
a-1/2+n+j\right) ^{\left( s\right) }\left( b+1/2+2j+s\right) ^{\left(
n-j-s\right) }.
\end{gather*}%
Now, denoting $x\allowbreak =\allowbreak a+1/2+2j$ and $y\allowbreak
=\allowbreak b+1/2+2j$, noticing that $x-y\allowbreak =\allowbreak a-b$ and
recalling (\ref{rozn3}), we see that the sum above is equal to $\left(
a-b\right) ^{\left( n-j\right) }$, proving (\ref{a1/2}). We prove (\ref{a3/2}%
) likewise. Now let us consider (\ref{ab}). We have%
\begin{gather*}
c_{n,j}(a,b;b,b)=\sum_{k=j}^{n}(-1)^{k-j}\frac{%
(a+b+n-1)^{(k)}(b+k)^{(n-k)}k!(b+j)^{(k-j)}}{%
k!(n-k)!(k-j)!(2b+j-1)^{(j)}(2b+2j)^{(k-j)}} \\
=\frac{(b+j)^{(n-j)}(a+b+n-1)^{(j)}}{(n-j)!(2b+j-1)^{(j)}(2b+2j)^{(n-j)}} \\
\times \sum_{k=j}^{n}(-1)^{k-j}\binom{n-j}{k-j}\frac{%
(a+b+n-1+j)^{(k-j)}(2b+2j)^{(n-j)}}{(2b+2j)^{(k-j)}}
\end{gather*}

\begin{gather*}
=\frac{(b+j)^{(n-j)}(a+b+n-1)^{(j)}}{(n-j)!(2b+j-1)^{(j)}(2b+2j)^{(n-j)}} \\
\times \sum_{s=0}^{n-j}(-1)^{s}\binom{n-j}{s}%
(a+b+n-1+j)^{(s)}(2b+2j+s)^{(n-j-s)} \\
=\frac{(b+j)^{(n-j)}(a+b+n-1)^{(j)}}{(n-j)!(2b+j-1)^{(j)}(2b+2j)^{(n-j)}}%
(-(a-b)-(n-j)-1)^{(n-j)}.
\end{gather*}

In the last equality, we used (\ref{rozn}). Now notice that obviously, we
have%
\begin{equation*}
(-1)^{n}(a)^{\left( n\right) }=(-a-n+1)^{(n)}.
\end{equation*}%
Let us consider now (\ref{ba}). We have%
\begin{gather*}
c_{n,j}\left( b,b;a,b\right) =\frac{1}{\left( n-j\right) !}%
\sum_{k=j}^{n}(-1)^{k-j}\binom{n-j}{k-j} \\
\times \frac{\left( 2b+n-1\right) ^{\left( k\right) }\left( b+k\right)
^{\left( n-k\right) }\left( b+j\right) ^{\left( k-j\right) }}{\left(
a+b+j-1\right) ^{\left( j\right) }\left( a+b+2j\right) ^{\left( k-j\right) }}
\\
=\frac{\left( b+j\right) ^{\left( n-j\right) }\left( 2b+n-1\right) ^{\left(
j\right) }}{\left( n-j\right) !\left( a+b+j-1\right) ^{\left( j\right)
}\left( a+b+2j\right) ^{\left( n-j\right) }}\times \\
\sum_{k=j}^{n}\left( -1\right) ^{k-j}\binom{n-j}{k-j}\left( 2b+n-1+j\right)
^{\left( k-j\right) }\left( a+b+k+j\right) ^{\left( n-k\right) }
\end{gather*}

\begin{gather*}
=\frac{\left( b+j\right) ^{\left( n-j\right) }\left( 2b+n-1\right) ^{\left(
j\right) }}{\left( n-j\right) !\left( a+b+j-1\right) ^{\left( j\right)
}\left( a+b+2j\right) ^{\left( n-j\right) }} \\
\times \sum_{s=0}^{n-j}\left( -1\right) ^{s}\binom{n-j}{s}\left(
2b+n-1+j\right) ^{\left( s\right) }\left( a+b+s+2j\right) ^{\left(
n-j-s\right) } \\
=\frac{\left( b+j\right) ^{\left( n-j\right) }\left( 2b+n-1\right) ^{\left(
j\right) }}{\left( n-j\right) !\left( a+b+j-b1\right) ^{\left( j\right)
}\left( a+b+2j\right) ^{\left( n-j\right) }} \\
\times \left( a+b+2j-2b-n+1-j\right) ^{\left( n-j\right) } \\
=\frac{\left( b+j\right) ^{\left( n-j\right) }\left( 2b+n-1\right) ^{\left(
j\right) }(a-b)^{\left( n\right) }}{\left( n-j\right) !\left(
a+b+j-b1\right) ^{\left( j\right) }\left( a+b+2j\right) ^{\left( n-j\right) }%
}.
\end{gather*}

To get the second part of (\ref{ab}) and the second part of (\ref{ba}) we
first recall use Proposition \ref{elem}(\ref{odwrkol}).

The fact that $c_{n,j}(a,a;b,b)\allowbreak =\allowbreak 0$ for odd $n-j$
follows symmetry of the distribution $h(x|a,a)$ and was noticed in (\ref%
{conn4}). Thus, it remains to consider the case when $n-j$ is even. In order
to get $n-j$ even we have to consider two cases. The first one is when $n$
is even and all $j\leq n$ must be even and the case when $n$ is odd and all $%
j\leq n$ must be odd. Now we have to refer to assertions of Lemma \ref{even}
and formulae (\ref{a1/2}) and (\ref{a3/2}), that were already proved. Let us
consider $n\allowbreak =\allowbreak 2m$ even. Then we have%
\begin{eqnarray*}
J_{2m}(x|a,a) &=&\frac{m!\left( a+m\right) ^{\left( m\right) }}{(2m)!}%
J_{m}(2x^{2}-1|a,1/2) \\
&=&\frac{m!\left( a+m\right) ^{\left( m\right) }}{(2m)!}%
\sum_{j=0}^{m}c_{m,j}(a,1/2;b,1/2)J_{j}(2x^{2}-1|b,1/2) \\
&=&\frac{m!\left( a+m\right) ^{\left( m\right) }}{(2m)!}%
\sum_{j=0}^{m}c_{m,j}(a,1/2;b,1/2)J_{2j}(x|b,b).
\end{eqnarray*}

Hence, 
\begin{gather*}
c_{2m,2j}(a,a;b,b)=\frac{m!\left( a+m\right) ^{\left( m\right) }}{(2m)!}%
c_{m,j}(a,1/2;b,1/2) \\
=\frac{m!\left( a+m\right) ^{\left( m\right) }}{(2m)!}\frac{\left(
1/2\right) ^{\left( m\right) }\left( a-b\right) ^{\left( m-j\right) }\left(
a-1/2+m\right) ^{\left( j\right) }\left( b-1/2+2j\right) }{\left( m-j\right)
!\left( 1/2\right) ^{\left( j\right) }\left( b+1/2+2j\right) ^{\left(
m-j\right) }\left( b-1/2+j\right) ^{\left( j+1\right) }}.
\end{gather*}

Similarly, when $n\allowbreak =\allowbreak 2m+1$, we argue as follows 
\begin{eqnarray*}
J_{2m+1}(x|a,a) &=&\frac{m!(a+m)^{(m+1)}}{(2m+1)!}xJ_{m}(2x^{2}-1|a,3/2) \\
&=&\frac{m!(a+m)^{(m+1)}}{(2m+1)!}%
x\sum_{j=0}^{m}c_{m,j}(a,3/2;b,3/2)J_{j}(2x^{2}-1|b,3/2) \\
&=&\frac{m!(a+m)^{(m+1)}}{(2m+1)!}%
\sum_{j=0}^{m}c_{m,j}(a,3/2;b,3/2)xJ_{j}(2x^{2}-1|b,3/2) \\
&=&\frac{m!(a+m)^{(m+1)}}{(2m+1)!}%
\sum_{j=0}^{m}c_{m,j}(a,3/2;b,3/2)J_{2j+1}(x|b,b).
\end{eqnarray*}%
Hence, we have 
\begin{gather*}
c_{2m+1,2j+1}(a,a;b,b)=\frac{m!(a+m)^{(m+1)}}{(2m+1)!}c_{m,j}(a,3/2;b,3/2) \\
=\frac{m!(a+m)^{(m+1)}}{(2m+1)!}\frac{\left( 3/2\right) ^{\left( m\right)
}\left( a-b\right) ^{\left( m-j\right) }\left( a+1/2+m\right) ^{\left(
j\right) }(b+1/2+2j)}{\left( m-j\right) !\left( 3/2\right) ^{\left( j\right)
}\left( b+3/2+2j\right) ^{\left( m-j\right) }\left( b+1/2+j\right) ^{\left(
j+1\right) }}.
\end{gather*}
\end{proof}

\begin{proof}[Proof of Corollary \protect\ref{Upr}]
In all simplifications below we will use identity: 
\begin{equation*}
\left( a\right) ^{\left( k+n\right) }=\left( a\right) ^{\left( k\right)
}\left( a+k\right) ^{\left( n\right) }.
\end{equation*}%
In order to get the first identity, we start with (\ref{ab}) 
\begin{gather*}
\frac{(b+j)^{(n-j)}(a-b)^{(n-j)}(a+b+n-1)^{(j)}}{%
(n-j)!(2b+j-1)^{(j)}(2b+2j)^{(n-j)}} \\
=\frac{1}{\left( n-j\right) !}\sum_{k=j}^{n}\left( -1\right) ^{k-j}\binom{n-j%
}{k-j}\frac{\left( a+b+n-1\right) ^{\left( k\right) }(a+m)^{(n-k)}\left(
b+j\right) ^{\left( k-j\right) }}{\left( 2b+j-1\right) ^{\left( j\right)
}\left( 2b+2j\right) ^{\left( k-j\right) }} \\
=\frac{\left( a+b+n-1\right) ^{\left( j\right) }}{\left( n-j\right) !\left(
2b+j-1\right) ^{\left( j\right) }\left( 2b+2j\right) ^{\left( n-j\right) }}
\\
\times \sum_{k=j}^{n}\left( -1\right) ^{k-j}\binom{n-j}{k-j}\left(
a+b+j+n-1\right) ^{\left( k-j\right) } \\
\times \left( b+j\right) ^{\left( k-j\right) }\left( a+j\right) ^{\left(
n-k\right) }\left( 2b+j+k\right) ^{\left( n-k\right) }.
\end{gather*}%
After cancelling common factors on both sides, we get%
\begin{eqnarray*}
&&(b+j)^{(n-j)}(a-b)^{(n-j)} \\
&=&\sum_{k=j}^{n}\left( -1\right) ^{k-j}\binom{n-j}{k-j}\left(
a+b+j+n-1\right) ^{\left( k-j\right) } \\
&&\times \left( b+j\right) ^{\left( k-j\right) }\left( a+j\right) ^{\left(
n-k\right) }\left( 2b+j+k\right) ^{\left( n-k\right) } \\
&=&\sum_{s=0}^{n-j}\left( -1\right) ^{s}\binom{n-j}{s}\left(
a+b+j+n-1\right) ^{\left( s\right) } \\
&&\times \left( b+j\right) ^{\left( s\right) }\left( a+j\right) ^{\left(
n-j-s\right) }\left( 2b+2j+s\right) ^{\left( n-j-s\right) }.
\end{eqnarray*}%
Now we set $a+j\allowbreak =\allowbreak x$, $b+j\allowbreak =\allowbreak y$,
and change $n-j\allowbreak $to $\allowbreak n$ and $s$ to $j$ and get (\ref%
{x-y}). The formula (\ref{y-x}) we prove in the similar way. The formulae (%
\ref{001}) and (\ref{002}) are justified using (\ref{01}) once firstly with $%
c_{n,j}(a,b;b,b)$ and $c_{n,j}(b,b;a,b)$ and secondly with $c_{n,j}(b,b;a,b)$
and $c_{n,j}\left( a,b;b,b\right) $. The formulae (\ref{dd1}) and (\ref{dd2}%
) are proven in a similar way using the following consequence of the formula
(\ref{conn3}):%
\begin{equation*}
c_{n,j}(a,a;b,b)\allowbreak =\allowbreak
\sum_{k=j}^{n}c_{n,k}(a,a;a,b)c_{k,j}(a,b,b,b).
\end{equation*}%
First, we take $n\allowbreak =\allowbreak 2m$ and proceed as follows:%
\begin{eqnarray*}
&&\frac{(2b+2j-1)\left( 2a+2m-1\right) ^{\left( j\right) }\left( a-b\right)
^{\left( m\right) }\left( b+j\right) ^{\left( m\right) }\left( a+m+j\right)
^{\left( m\right) }}{\left( m\right) !\left( 2b+j-1\right) ^{\left(
2m+1+j\right) }} \\
&=&\frac{(2b+2j-1)}{(2m)!}\sum_{k=j}^{2m+j}\left( -1\right) ^{k-j}\binom{2m}{%
k-j} \\
&&\times \frac{\left( 2a+2m+j-1\right) ^{\left( k\right) }\left( a+k\right)
^{\left( 2m+j-k\right) }\left( b+j\right) ^{\left( k-j\right) }}{\left(
2b+j-1\right) ^{\left( k+1\right) }}.
\end{eqnarray*}%
Now, we cancel out $(2b+2j-1)$ on both sides. Further, we change the index
of summation, setting $s=\allowbreak k-j$, then we multiply both sides by $%
\left( 2b+j-1\right) ^{\left( 2m+1+j\right) }$ and divide both sides by $%
\left( 2a+2m-1\right) ^{\left( j\right) }$. We get then:%
\begin{eqnarray*}
&&\frac{\left( 2m\right) !}{m!}\left( a-b\right) ^{\left( m\right) }\left(
b+j\right) ^{\left( m\right) }\left( a+m+j\right) ^{\left( m\right) } \\
&=&\sum_{s=0}^{2m}\left( -1\right) ^{s}\binom{2m}{s}\left( 2a+2m+2j-1\right)
^{\left( s\right) }\left( a+j+s\right) ^{\left( 2m-s\right) } \\
&&\times \left( b+j\right) ^{\left( s\right) }\left( 2b+2j+s\right) ^{\left(
2m-s\right) }.
\end{eqnarray*}%
The last step is to define $x\allowbreak =\allowbreak a+j$ and $y\allowbreak
=\allowbreak b+j$ and notice that $x-y\allowbreak =\allowbreak a-b.$

In order to get (\ref{dd2}), we use the fact that for $n\allowbreak
=\allowbreak 2m+1+j$ we have $c_{n,j}(a,a;b,b)\allowbreak =\allowbreak 0$
for all $m\geq 0$ and $j\geq 0$ and then proceed likewise.

In order to get (\ref{aaababaa}) we start with the obvious identity 
\begin{equation*}
\sum_{k=j}^{n}c_{nk}(a,a;a,b)c_{k,j}(a,b;a,a)=%
\begin{cases}
1\text{,} & \text{if }n=j\geq 0\text{;} \\ 
0\text{,} & \text{if }n>j\geq 0\text{.}%
\end{cases}%
\end{equation*}%
We insert (\ref{ab}) and (\ref{ba}) and first put all expressions depending
only on $n$ and $j$ in the form of the sum, getting%
\begin{eqnarray*}
&&\left( -1\right) ^{n-j}\frac{\left( b+j\right) ^{\left( n-j\right) }\left(
a+b+2j-1\right) }{\left( n-j\right) !}\sum_{k=j}^{n}\binom{n-j}{k-j}\times \\
&&\frac{\left( a-b\right) ^{\left( n-k\right) }\left( a+b+n-1\right)
^{\left( k\right) }\left( b-a\right) ^{k-j}\left( 2b+2k-1\right) \left(
2b+k-1\right) ^{\left( j\right) }}{\left( 2b+k-1\right) ^{\left( n+1\right)
}\left( a+b+j-1\right) ^{\left( k+1\right) }} \\
&=&\left( -1\right) ^{n-j}\frac{\left( b+j\right) ^{\left( n-j\right)
}\left( a+b+2j-1\right) }{\left( n-j\right) !}\sum_{k=j}^{n}\binom{n-j}{k-j}%
\times \\
&&\frac{\left( a-b\right) ^{\left( n-k\right) }\left( a+b+n-1\right)
^{\left( k\right) }\left( b-a\right) ^{\left( k-j\right) }\left(
2b+2k-1\right) }{\left( 2b+k+j-1\right) ^{\left( n-j+1\right) }\left(
a+b+j-1\right) ^{\left( k+1\right) }}
\end{eqnarray*}%
\begin{eqnarray*}
&=&\left( -1\right) ^{n-j}\frac{\left( b+j\right) ^{\left( n-j\right)
}\left( a+b+2j-1\right) }{\left( n-j\right) !}=\sum_{s=0}^{n-j}\binom{n-j}{s}%
\times \\
&&\frac{\left( a-b\right) ^{\left( n-j-s\right) }\left( a+b+n-1\right)
^{\left( j+s\right) }\left( b-a\right) ^{\left( s\right) }\left(
2b+2j+2s-1\right) }{\left( 2b+2j+s-1\right) ^{\left( n-j+1\right) }\left(
a+b+j-1\right) ^{\left( j+s+1\right) }}.
\end{eqnarray*}%
Now we set $n-j\allowbreak =\allowbreak m$. We get then:%
\begin{eqnarray*}
&&\left( -1\right) ^{m}\frac{\left( b+j\right) ^{\left( m\right) }\left(
a+b+2j-1\right) \left( a+b+m+j-1\right) ^{\left( j\right) }}{\left( m\right)
!\left( a+b+j-1\right) ^{\left( j\right) }}\times \\
&&\sum_{s=0}^{m}\binom{m}{s}\frac{\left( a-b\right) ^{\left( m-s\right)
}\left( a+b+m+2j-1\right) ^{\left( s\right) }\left( b-a\right) ^{\left(
s\right) }\left( 2b+2j+2s-1\right) }{\left( 2b+2j+s-1\right) ^{\left(
m+1\right) }\left( a+b+2j-1\right) ^{\left( s+1\right) }}.
\end{eqnarray*}%
Further we cancel out $\left( a+b+2j-1\right) $ and denote $x\allowbreak
=\allowbreak a+j$ and $y\allowbreak =\allowbreak b+j$. We get then:%
\begin{eqnarray*}
&&\left( -1\right) ^{m}\frac{\left( y\right) ^{\left( m\right) }\left(
x+y+m-j-1\right) ^{\left( j\right) }}{\left( m\right) !\left( x+y-j-1\right)
^{\left( j\right) }}\times \\
&&\sum_{s=0}^{m}\binom{m}{s}\frac{\left( x-y\right) ^{\left( m-s\right)
}\left( x+y+m-1\right) ^{\left( s\right) }\left( y-x\right) ^{\left(
s\right) }\left( 2y+2s-1\right) }{\left( 2y+s-1\right) ^{\left( m+1\right)
}\left( x+y\right) ^{\left( s\right) }}.
\end{eqnarray*}%
Finally, we split $\left( 2y+s-1\right) ^{\left( m+1\right) }$ to $\left(
2y+s-1\right) ^{\left( s\right) }\left( 2y+2s-1\right) ^{\left( m-s+1\right)
}$ and cancel our $\left( 2y+2s-1\right) .$

We start with the identity 
\begin{equation*}
\sum_{k=j}^{n}c_{2n,2k}(a,a;b,b)c_{2k,2j}(b,b;a,a)=%
\begin{cases}
0\text{,} & \text{if }n>j\geq 0\text{;} \\ 
1\text{,} & \text{if }n\geq j\geq 0\text{.}%
\end{cases}%
\end{equation*}%
We have, after inserting (\ref{aabb})%
\begin{eqnarray*}
&&\sum_{k=j}^{n}\frac{(2a+4j-1)\left( b-a\right) ^{\left( k-j\right) }\left(
a+2j\right) ^{\left( k-j\right) }\left( 2b+2k-1\right) ^{\left( 2j\right)
}\left( b+j+k\right) ^{\left( k-j\right) }}{(k-j)!\left( 2a+2j-1\right)
^{\left( 1+2k\right) }}\times \\
&&\frac{\left( 2b+4k-1\right) \left( a-b\right) ^{\left( n-k\right) }\left(
b+2k\right) ^{\left( n-k\right) }\left( 2a+2n-1\right) ^{\left( 2k\right)
}\left( a+k+n\right) ^{\left( n-k\right) }}{\left( n-k\right) !\left(
2b+2k-1\right) ^{\left( 2n+1\right) }}.
\end{eqnarray*}%
Now, we try to put all expressions that depend on only $n$ and $j$ outside
the sum. So we have further%
\begin{eqnarray*}
&&\frac{\left( 2a+2n-1\right) ^{\left( 2j\right) }}{\left( n-j\right)
!\left( 2a+2j-1\right) ^{\left( 2j\right) }}\sum_{k=j}^{n}\binom{n-j}{k-j}%
\frac{\left( b-a\right) ^{\left( k-j\right) }\left( a+2j\right) ^{\left(
k-j\right) }\left( b+j+k\right) ^{\left( k-j\right) }}{\left( 2a+4j\right)
^{\left( 2k-2j\right) }}\times \\
&&\frac{\left( 2b+4k-1\right) \left( 2a+2n+2j-1\right) ^{\left( 2k-2j\right)
}\left( a+k+n\right) ^{\left( n-k\right) }\left( a-b\right) ^{\left(
n-k\right) }\left( b+2k\right) ^{\left( n-k\right) }}{\left(
2b+2k+2j-1\right) ^{\left( 2n-2j+1\right) }} \\
&=&\frac{\left( 2a+2n-1\right) ^{\left( 2j\right) }}{\left( n-j\right)
!\left( 2a+2j-1\right) ^{\left( 2j\right) }}\sum_{k=j}^{n}\binom{n-j}{k-j}%
\frac{\left( b-a\right) ^{\left( k-j\right) }\left( a+2j\right) ^{\left(
k-j\right) }\left( b+j+k\right) ^{\left( k-j\right) }}{\left( 2a+4j\right)
^{\left( 2k-2j\right) }}\times \\
&&\frac{\left( 2a+2n+2j-1\right) ^{\left( 2k-2j\right) }\left( a+k+n\right)
^{\left( n-k\right) }\left( a-b\right) ^{\left( n-k\right) }\left(
b+2k\right) ^{\left( n-k\right) }}{\left( 2b+2k+2j-1\right) ^{\left(
2k-2j\right) }\left( 2b+4k\right) ^{\left( 2n-2k\right) }}.
\end{eqnarray*}%
Now, we denote $s\allowbreak =\allowbreak k-j$ and $m\allowbreak
=\allowbreak n-j$. We have%
\begin{eqnarray*}
&&\frac{\left( 2a+2m+2j-1\right) ^{\left( 2j\right) }}{\left( m\right)
!\left( 2a+2j-1\right) ^{\left( 2j\right) }}\sum_{s=0}^{m}\binom{m}{s}\frac{%
\left( b-a\right) ^{\left( s\right) }\left( a+2j\right) ^{\left( s\right)
}\left( b+2j+s\right) ^{\left( s\right) }}{\left( 2a+4j\right) ^{\left(
2s\right) }}\times \\
&&\frac{\left( 2a+2m+4j-1\right) ^{\left( 2s\right) }\left( a+2j+s+m\right)
^{\left( m-s\right) }\left( a-b\right) ^{\left( m-s\right) }\left(
b+2j+2s\right) ^{\left( m-s\right) }}{\left( 2b+2s+4j-1\right) ^{\left(
2s\right) }\left( 2b+4j+4s\right) ^{\left( 2m-2s\right) }}.
\end{eqnarray*}%
Now, we set $x\allowbreak =\allowbreak a+2j$ and $y\allowbreak =\allowbreak
b+2j$ and notice that $x-y\allowbreak =\allowbreak a-b$. hence we get%
\begin{eqnarray*}
&&\frac{\left( 2x+2m-1\right) ^{\left( 2j\right) }}{\left( m\right) !\left(
2x-1\right) ^{\left( 2j\right) }}\sum_{s=0}^{m}\binom{m}{s}\frac{\left(
y-x\right) ^{\left( s\right) }\left( x\right) ^{\left( s\right) }\left(
y+s\right) ^{\left( s\right) }}{\left( 2x\right) ^{\left( 2s\right) }}\times
\\
&&\frac{\left( 2x+2m-1\right) ^{\left( 2s\right) }\left( x+m+s\right)
^{\left( m-s\right) }\left( x-y\right) ^{\left( m-s\right) }\left(
y+2s\right) ^{\left( m-s\right) }}{\left( 2y+2s-1\right) ^{\left( 2s\right)
}\left( 2y+4s\right) ^{\left( 2m-2s\right) }}.
\end{eqnarray*}%
Now, we apply the following identity%
\begin{equation*}
\left( 2z\right) ^{\left( 2t\right) }=4^{t}\left( z\right) ^{\left( s\right)
}\left( z+1/2\right) ^{\left( s\right) },
\end{equation*}%
and get the assertion.
\end{proof}

\end{document}